\documentclass[12pt]{amsart}

\usepackage{amsmath} \usepackage{amssymb} \usepackage{amsthm}
\usepackage{amscd}
\usepackage[all]{xy}
\usepackage{enumerate}
\usepackage{mathrsfs}
\usepackage{multirow}
\usepackage{bigdelim}
\usepackage[bookmarks=false]{hyperref}

\def\tb{\textbf}

\def\a{{\alpha}}

\def\G{{\mathcal{G}}}

\def\O{{\mathcal{O}}}

\def\Z{{\mathbb{Z}}}
\def\N{{\mathbb{N}}}
\def\Q{{\mathbb{Q}}}
\def\R{{\mathbb{R}}}

\def\ol{\overline}

\def\Pic{{\mathrm{Pic}}}

\theoremstyle{plain}
\newtheorem{thm}{Theorem}[section]
\newtheorem{cor}[thm]{Corollary}
\newtheorem{prop}[thm]{Proposition}
\newtheorem{conj}[thm]{Conjecture}

\newtheorem{lem}[thm]{Lemma}
\theoremstyle{definition}
\newtheorem{defn}[thm]{Definition}

\theoremstyle{remark}
\newtheorem{rem}[thm]{Remark}

\newtheorem{notation}[thm]{Notation}

\newtheorem{cln}[thm]{Claim}
\newtheorem*{acknowledgement}{Acknowledgments}

\begin{document}

\title[Abundance for 3-folds with non-trivial Albanese maps]{Abundance for 3-folds with non-trivial Albanese maps in positive characteristic}

\address{Lei Zhang\\School of Mathematical Science\\University of Science and Technology of China\\Hefei 230026, P.R.China.}
\email{zhleimath@163.com, zhlei18@ustc.edu.cn}
\author{Lei Zhang}
\maketitle

\begin{abstract}
In this paper, we prove abundance for 3-folds with non-trivial Albanese maps, over an algebraically closed field of characteristic $p > 5$.
\end{abstract}

\section{Introduction}
Over an algebraically closed field of characteristic $p > 5$, existence of log minimal models of 3-folds has been proved by Birkar, Hacon and Xu (\cite{Bir16} \cite{HX15}); and
log abundance has been proved for minimal klt pairs $(X,B)$ when $K_X + B$ is big or $B$ is big (\cite{Bir16} \cite{CTX15} \cite{Xu15}), and when $X$ is non-uniruled and has non-trivial Albanese map (\cite{Zh16c}).

This paper aims to prove abundance for 3-folds with non-trivial Albanese maps.
\begin{thm}\label{abundance}
Let $X$ be a klt, $\Q$-factorial, projective minimal 3-fold, over an algebraically closed field  $k$ with $\mathrm{char}~k = p >5$. Assume that
the Albanese map $a_X: X \rightarrow A_X$ is non-trivial.
Then $K_X$ is semi-ample.
\end{thm}

Precisely we prove log abundance in some cases.
\begin{thm}\label{abundance-log}
Let $(X, B)$ be a klt, $\Q$-factorial, projective minimal pair of dimension 3, over an algebraically closed field  $k$ with $\mathrm{char}~k = p >5$. Assume that
the Albanese map $a_X: X \rightarrow A_X$ is non-trivial. Denote by $f: X \to Y$ the fibration arising from the Stein factorization of $a_X$ and by $X_{\eta}$ the generic fiber of $f$.
Assume moreover that $B=0$ if

(1) $\dim Y = 2$ and $\kappa(X_{\eta}, (K_X + B)|_{X_{\eta}})=0$, or

(2) $\dim Y = 1$ and $\kappa(X_{\eta}, (K_X + B)|_{X_{\eta}})=1$.\\
Then $K_X + B$ is semi-ample.
\end{thm}

Throughout of this paper, as $f: X \to Y$ frequently appears as a projective morphism of varieties, we denote by $\eta$ (resp. $\bar{\eta}$) the generic (resp. geometric generic) point of $Y$ and by $X_{\eta}$ (resp. $X_{\bar{\eta}}$) the generic (resp. geometric generic) fiber of $f$. Moreover we say $f: X \to Y$ is a \emph{fibration} if $f_*\O_X = \O_Y$.

To study abundance for varieties with non-trivial Albanese maps, it is necessary to study the following conjecture on subadditivity of Kodaira dimensions.
\begin{conj}[Iitaka conjecture] Let $f:X\rightarrow Y$ be a fibration between two smooth projective varieties over an algebraically closed field $k$, with $\dim X = n$ and $\dim Y = m$. Then
  $$C_{n,m}: \kappa(X)\geq \kappa(Y) + \kappa(X_{\eta}, K_{X_\eta}).$$
\end{conj}
The divisor $K_{X_\eta} = K_X|_{X_\eta}$ is a Cartier divisor corresponding to the dualizing sheaf of $X_{\eta}$, which is invertible since $X_{\eta}$ is regular.
In characteristic zero, since the geometric generic fiber $X_{\bar{\eta}}$ is smooth over $k(\ol\eta)$, so $\kappa(X_{\bar{\eta}}) = \kappa(X_{\eta}, K_{X_\eta})$.
In positive characteristic, $X_{\ol\eta}$ is not necessarily smooth over $k(\ol\eta)$,
if $X_{\ol\eta}$ has a smooth projective birational model $\tilde{X}_{\bar{\eta}},$\footnote{Existence of smooth resolution of singularities has been proved in dimension $\leq 3$ by \cite{CP08}.} then by \cite[Corollary 2.5]{CZ15}
$$\kappa(X_{\eta}, K_{X_\eta})= \kappa(X_{\bar{\eta}}, K_{X_{\bar{\eta}}}) \geq \kappa(\tilde{X}_{\bar{\eta}}).$$
We refer to \cite{Zh16a} for more discussions on this conjecture. In this paper we prove the following theorem, which implies Theorem \ref{abundance-log} after combined with the results of \cite{Zh16c}.

\begin{thm}[= Theorem \ref{Iit-conj}]\label{0}
Let $f: X \to Y$ be a fibration from a $\Q$-factorial projective 3-fold to a smooth projective variety of dimension $1$ or $2$, over an algebraically closed field  $k$ with $\mathrm{char}~k = p >5$. Let $B$ be an effective $\Q$-divisor on $X$ such that $(X,B)$ is klt.
Assume that $Y$ is of maximal Albanese dimension, and assume moreover that
\begin{itemize}
\item[$\spadesuit$]
if $\kappa(X_{\eta}, K_{X_{\eta}} + B_{\eta}) = \dim X/Y - 1$, then $B$ does not intersect the generic fiber $X_{\xi}$ of the relative Iitaka fibration
$I: X \dashrightarrow Z$ induced by $K_X +B$ on $X$ over $Y$.
\end{itemize}
Then
$$\kappa(X, K_X + B) \geq \kappa(X_{\eta}, K_{X_{\eta}} + B_{\eta}) + \kappa(Y).$$
\end{thm}
\begin{rem}
As $\mathrm{char}~k = p >5$, for a fibration $h: X \to Z$, if the generic fiber is a curve with arithmetic genus one, then the geometric generic fiber must be a smooth elliptic curve (Proposition \ref{fib-pa=1}). So in case $\kappa(X_{\eta}, K_{X_{\eta}} + B_{\eta}) = \dim X/Y - 1$, the assumption $\spadesuit$ guarantees that the relative Iitaka fibration $I: X \dashrightarrow Z$ is fibred by elliptic curves. The advantage of an elliptic fibration is the canonical bundle formula: if $h: W \to Z$ is a flat relative minimal elliptic fibration then there exists an effective divisor $B_Z$ on $Z$ such that $K_W \sim_{\Q} h^*(K_Z + B_Z)$ (\cite[3.2]{CZ15}). Canonical bundle formula is the key technique to study log abundance (\cite{KMM94}).
Here we mention that in positive characteristic, only under some very strong conditions, $B_Z$ has been proved to be effective (\cite[Lemma 6.7]{CTX15}, \cite[Theorem 3.18]{Ej15}, \cite[Theorem B]{DS15}).
\end{rem}

The result above has been proved when the geometric generic fiber $X_{\bar{\eta}}$ is smooth and $B=0$ (\cite{CZ15, EZ16}), which was used in \cite{Zh16c} to prove log abundance for non-uniruled 3-folds with non-trivial Albanese maps.
To study the uniruled case, we have to treat fibrations with singular geometric generic fibers. For a separable fibration with possibly singular geometric generic fiber, $C_{3,m}$ has been proved by the author \cite{Zh16a} when $K_X + B$ is $f$-big, $K_Y$ is big and $Y$ is of maximal Albanese dimension.
To prove Theorem \ref{0}, the essentially new results are the following two theorems.
\begin{thm}[= Theorem \ref{fib-to-ab-var}]\label{1}
Let $(X,B)$ be a projective $\Q$-factorial klt pair of dimension 3, over an algebraically closed field  $k$ with $\mathrm{char}~k = p >5$. Let $f: X \to Y=A$ be a fibration to an elliptic curve or a simple abelian surface. Assume that $K_X + B$ is $f$-big. Then
$$\kappa(K_X + B) \geq \kappa(X_{\eta}, K_{X_{\eta}} + B_{\eta}).$$
\end{thm}

\begin{thm}[= Theorem \ref{fib-to-curve}]\label{2}
Let $(X,B)$ be a $\Q$-factorial klt pair of dimension 3, over an algebraically closed field  $k$ with $\mathrm{char}~k = p >5$. Let $f: X \to Y$ be a fibration to a normal curve $Y$ of genus $g(Y) \geq 1$. Assume $\kappa(X_{\eta}, (K_X + B)|_{X_{\eta}}) = 0$. Then
$$\kappa(X, K_X + B) \geq \kappa(Y).$$

If moreover $K_X +B$ is nef then it is semi-ample.
\end{thm}

We summarize the techniques to study $C_{3,m}$ which appeared in the papers \cite{Pa14, Ej15, Zh16a, EZ16}, then explain the ideas of the proof of Theorem \ref{1} and \ref{2}.

\smallskip

(1) Positivity results. Let $f: X \to Y$ be a separable fibration of normal projective varieties. In positive characteristic, Parakfalvi \cite{Pa14} first proved that for sufficiently divisible $n$, the sheaf $f_*\O_X(n(K_{X/Y} + B))$ is weakly positive under the condition that $(X_{\ol\eta}, B_{\ol\eta})$ is strongly $F$-regular and $K_{X/Y} + B$ is $f$-ample, then Ejiri \cite{Ej15} reproved it and made some generalizations using different method. This result may fail when $X_{\ol\eta}$ is singular. But in dimension three and over an algebraically closed field $k$ with $\mathrm{char}~k = p >5$, we can take advantage of minimal model theory. Under the condition that $K_X +B$ is nef, relatively big and semi-ample over $Y$, the author \cite{Zh16a} (or \cite{PSZ13} under stronger conditions) proved that for sufficiently divisible $n, g>0$, the sheaf $F_Y^{g*}(f_*\O_X(n(K_{X/Y} + B))\otimes \omega_Y^{n-1})$ contains a non-zero weakly positive sub-sheaf $V_n$, which plays a similar role as $f_*\O_X(n(K_{X/Y} + B))$ does in studying subadditivity of Kodaira dimensions.

If $f: X \to Y$ is a fibration from a 3-fold to a smooth curve, when $\kappa(X_{\ol\eta}) = 1$ the relative Iitaka fibration of $X$ over $Y$ is fibred by elliptic curves since $p>5$ (Proposition \ref{fib-pa=1}). By canonical bundle formula (\cite[3.2]{CZ15}) and minimal model theory, one can reduce to a pair $(Z, B_Z)$ of dimension two with $K_Z + B_Z$ being relatively big over $Y$. Then one can show that $f_*\omega_{X/Y}^n$ contains a non-zero nef sub-sheaf $V_n$ (\cite[Theorem 2.8]{EZ16}).

Using these positivity results above, one can usually treat the case when $K_Y$ is big or the case when $\det V_n$ is big.
Note that this approach only requires $K_X + B$ to be nef (not necessarily klt). This paper also treats inseparable fibrations,  which can be reduced to a separable fibration of a pair $(X', B')$ not necessarily klt by applying foliation theory (Proposition \ref{red-to-sep}).

\smallskip

To treat the case when $Y$ is an elliptic curve and $\deg V_n = 0$, we have the following two approaches.

(2.1) Trace maps of relative Frobenius. Ejiri \cite[Theorem 1.7]{Ej15} introduced a clever trick as follows.
First there exists an isogeny $\pi: Y' \to Y$ such that $\pi^*V_n = \bigoplus L_i$ where $L_i \in \mathrm{Pic}^0(Y')$ by \cite[1.4. Satz]{LS77} and \cite[Corollary 2.10]{Oda71}. Then by applying trace maps of relative Frobenius, one gets many relations of $L_i$, from which one can prove that every $L_i$ is torsion in $\mathrm{Pic}^0(Y')$. This indicates that for sufficiently divisible $N$, the line bundle $\O(NK_X)$ has many global sections.

This method heavily relies on the smoothness of the geometric generic fiber $X_{\ol\eta}$, it was applied to prove $C_{3,1}$ when $X_{\ol\eta}$ is smooth and either $K_{X}$ is  $f$-big (\cite[Thm. 1.7]{Ej15}) or is $f$-$\Q$-trivial (\cite[Sec. 3]{EZ16}).

(2.2) Adjunction formula. Granted the isogeny $\pi: Y' \to Y$ and the splitting $\pi^*V_n = \bigoplus L_i$, first by applying covering theorem (Theorem \ref{ct}) one can construct effective divisors $D_i \sim NK_X + P_i$ for some $P_i \in f^*\mathrm{Pic}^0(Y)$ and an integer $N$. Then applying adjunction formula to different components of $D_i$ and log abundance for surfaces (\cite{Ta14}), one can get some relations of $P_i$ which conclude that $P_i$ are torsion.

This approach was used in \cite[Sec. 4]{EZ16} and \cite{Zh16b}. In fact it applies once granted positivity results as in (1) without requiring $X_{\ol\eta}$ to be smooth.

\medskip

Now we explain the ideas of the proof of Theorem \ref{1} and \ref{2}.
\smallskip

(3) To prove Theorem \ref{1}, we consider the cohomological locus
$$V^0(f_*\O_X(n(K_{X} + B))) = \{L \in \mathrm{Pic}^0(A)| h^0(A, f_*\O_X(n(K_{X} + B))\otimes L) >0\}.$$
If $\dim (V^0(f_*\O_X(n(K_{X} + B)))) > 0$, then $V^0(f_*\O_X(n(K_{X} + B)))$ generates $\mathrm{Pic}^0(A)$ since $\mathrm{Pic}^0(A)$ is simple, and it is easy to show that $\kappa(X, K_{X} + B) \geq \dim A$.
For the remaining case $\dim V^0(f_*\O_X(n(K_{X} + B))) \leq 0$, we follow the approach in (2.2). The key point is to find at least two effective divisors $D_i \sim n(K_X + B) + P_i$ for some $P_i \in f^*\mathrm{Pic}^0(A)$. We try to find a subsheaf $\mathcal{F}$ of $f_*\O_X(n(K_{X} + B))$ and an isogeny $\pi: A' \to A$ such that $\pi^*\mathcal{F} = \bigoplus L_i$ where $L_i \in \mathrm{Pic}^0(A')$. The sheaf $\mathcal{F}$ is obtained as the image of the trace map
$$Tr_{X,B}^{e, n}: f_*(F_{X*}^{e}\O_X((1-p^e)(K_X + B)) \otimes \O_X(n(K_X + B)))  \to f_*\O_X(n(K_X + B)).$$
We apply Frobenius amplitude to show that $\mathcal{F}$ satisfies a property slightly weaker than generic vanishing (Theorem \ref{CGG}), then apply ``killing cohomology'' (\cite[Prop. 12 and Sec. 9]{Se58b} or \cite[Lemma 1.3]{HPZ17}) to get an isogeny $\pi: A' \to A$, some $P_i\in \mathrm{Pic}^0(A')$ and a generically surjective homomorphism $\bigoplus_iP_i \to \pi^*\mathcal{F}$.

\medskip

(4) To prove Theorem \ref{2}, we can assume $K_X + B$ to be $f$-nef by replacing $(X,B)$ with a relative minimal model, then for a sufficiently divisible integer $k>0$, $k(K_X + B) = f^*L$ for some nef $L \in \mathrm{Pic}(Y)$ (\cite{Ta15b}). It suffices to show that either $\deg L > 0$ or $L$ is torsion. For the case $\deg L = 0$, we use the strategy of \cite[Theorem 8.10]{Ba01}. First we construct a nef divisor $D = aH - bF$ on $X$ with $\nu(D) = 2$ where $H$ is an ample divisor and $F$ is a general fiber of $f$, second we prove that there exists a semi-ample divisor $D' \equiv D$, which induces a fibration $g: X \to Z$ to a surface, finally we can show the restriction on the generic fiber $k(K_X+B)|_G = k(K_G+ B|_G) \sim 0$, which implies that $L$ is torsion.

\begin{rem}
The crucial results to be used include minimal model theory of 3-folds \cite{Bir16, BW14, HX15, Xu15}, abundance for surfaces \cite{Ta14, Ta15b} and results on positivity and subadditivity of Kodaira dimensions in \cite{Zh16a}. We try to make the proof self-contained. But for some cases we only sketch the proof and refer to \cite{EZ16} and \cite{Zh16c} for details.
\end{rem}

This paper is organized as follows. In Section \ref{sec-pre}, we collect some useful results and study inseparable fibrations as preparations. Section \ref{sec-shf-ab-var} is devoted to studying sheaves on abelian varieties. In Section \ref{sec-pf-Iitaka}, we study subadditivity of Kodaira dimensions. Finally in Section \ref{sec-pf-abundance} we finish the proof of abundance in Theorem \ref{abundance-log}.

\medskip

\textbf{Conventions:}
Sometimes we do not distinguish between line bundles and Cartier divisors, and abuse the notation additivity and tensor product if no confusion occurs.

Let $X$ be a normal projective variety. Denote by $\mathrm{Wdiv}(X)$ the set of Weil divisors and by $\mathrm{Cdiv}(X)$ the set of Cartier divisors on $X$. Assume $\mathbb{K} = \Q$ or $\R$. The divisors in $\mathrm{Wdiv}(X) \otimes_{\Z} \mathbb{K}$ are called $\mathbb{K}$-divisors, and the ones in $\mathrm{Cdiv}(X) \otimes_{\Z} \mathbb{K}$ are called $\mathbb{K}$-Cartier $\mathbb{K}$-divisors. We use $\equiv$ for the numerical equivalence relation in $\mathrm{Cdiv}(X) \otimes_{\Z} \mathbb{K}$.

Let $D$ be a $\Q$-divisor on a normal variety $X$. The Weil index of $D$ is the minimal positive integer $l$ such that $lD$ is integral. If $D$ is $\Q$-Cartier, the Cartier index is defined similarly. We use $\sim$ (resp. $\sim_{\mathbb{Q}}$) for linear (resp. $\mathbb{Q}$-linear) equivalence between Cartier (resp. $\mathbb{Q}$-Cartier) divisors and line bundles. For two ($\Q$-)divisors $D_1, D_2$ on $X$, if $D_1|_{X^{sm}} \sim D_2|_{X^{sm}}$ (resp. $D_1|_{X^{sm}} \sim_{\Q} D_2|_{X^{sm}}$), we also denote $D_1 \sim D_2$ (resp. $D_1 \sim_{\Q} D_2$).

Let $X$ be a normal variety and $D$ a Weil divisor on $X$. Then $\mathcal{O}_X(D)$ is a subsheaf of the constant sheaf $K(X)$ of rational functions, with the stalk at a point $x$ being defined by
$$\mathcal{O}_X(D)_x:=\{f \in K(X)| ((f) + D)|_U \geq 0 ~\mathrm{for~some~open~set}~U~\mathrm{containing}~x\}.$$

For notions in minimal model theory such as lc, klt dlt pairs, flip and divisorial contraction and so on, we refer to \cite{Bir16}.

For a variety $X$, we usually use $F_X^e: X \to X$ to denote the $e^{\mathrm{th}}$ iteration of absolute Frobenius, we sometimes use $X^e$ for the origin scheme of $F_X^e$ to avoid confusions.

Let $\varphi:X \to T$ be a morphism of schemes and let $T'$ be a $T$-scheme.
Then we denote by $X_{T'}$ the fiber product
$X\times_{T}T'$. For a divisor $D$ on $X$ (resp. an $\O_X$-module $\G$), the pullback of $D$ (resp. $\G$) to $X_{T'}$
is denoted by $D_{T'}$ or $D|_{X_{T'}}$ (resp. $\G_{T'}$ or $\G|_{X_{T'}}$) if it is well-defined.

If $X$ is a projective variety in dimension $\leq 3$, it always has a smooth birational model $\tilde{X}$, then $\kappa(X) := \kappa(\tilde{X})$ which is birational invariant.

\begin{acknowledgement}
Part of the work was done during the period of the author visiting Xiamen University and Institute for Mathematical Sciences of National University of Singapore, he would like to thank Prof. Wenfei Liu and Deqi Zhang for their hospitalities and support. The author also thanks Hiromu Tanaka, Yi Gu, Zsolt Patakfalvi, Zhi Jiang and Yong Hu for some useful communications and discussions. He is very grateful to the anonymous referee, who kindly points out one serious mistake in the original version and many other inaccuracies and gives many valuable and explicit suggestions to improve the proof and the writing.
This work is supported by grant
NSFC (No. 11771260).
\end{acknowledgement}

\section{Preliminaries}\label{sec-pre}
In this section we collect some technical results which will be used in the paper.

\subsection{Separability of fibrations}
\begin{prop}\label{sep-fib}
Let $f: X \to Y$ be a fibration of normal varieties over an algebraically closed field $k$ of characteristic $p>0$. Then

(1) $f$ is separable if and only if $X_{\bar{\eta}}$ is reduced, and if and only if $X_{\bar{\eta}}$ is integral;

(2) if $\dim Y = 1$ then $f$ is separable.
\end{prop}
\begin{proof}
Since $f$ is a fibration we have $H^0(\O_{X_{\eta}}) = K(Y)$, and since $X_{\eta}$ is normal we can show $K(Y)$ is algebraically closed in $K(X)$. Then the assertion (1) follows from applying \cite[Sec. 3.2.2 Cor. 2.14 and Prop. 2.15]{Liu10}. The assertion (2) is \cite[Lemma 7.2]{Ba01}.
\end{proof}

\subsection{Relative Fujita Vanishing} The following result is \cite[Theorem 1.5]{Ke03}.
\begin{lem}\label{rel-fujita}
Let $f: X \rightarrow Y$ be a projective morphism over a Noetherian scheme, $H$ an $f$-ample line bundle and $\mathcal{F}$ a coherent sheaf on $X$. Then there exists a positive integer $N$ such that, for every $n >N$ and every nef line bundle $L$,
$$R^if_*(\mathcal{F}\otimes H^n \otimes L) = 0, \mathrm{~if~} i>0.$$
\end{lem}

\subsection{Covering Theorem}
The result below is [\cite{Iit82}, Theorem 10.5] when $X$ and $Y$ are both smooth, and the proof also applies when varieties are normal.
\begin{thm}\textup{(\cite[Theorem 10.5]{Iit82})}\label{ct}
Let $f\colon X \rightarrow Y$ be a proper surjective morphism between complete normal varieties.
If $D$ is a Cartier divisor on $Y$ and $E$ an effective $f$-exceptional divisor on $X$, then
$$\kappa(X, f^*D + E) = \kappa(Y, D).$$
\end{thm}

As a corollary we get the following useful result, which also appeared in \cite{EZ16}.
\begin{lem}\emph{(}\cite[Lemma 2.3]{EZ16}\emph{)}\label{inj-pic}
Let $g: W \rightarrow Y$ be a surjective projective morphism between projective varieties. Assume $Y$ is normal and let $L_1,L_2 \in \Pic(Y)$ be two line bundles on $Y$. If $g^*L_1 \sim_{\mathbb{Q}} g^*L_2$ then $L_1 \sim_{\mathbb{Q}} L_2$.
\end{lem}
\begin{proof}
Let $L = L_1 \otimes L_2^{-1}$. Denote by $\sigma: W' \to W$ the normalization and let $g'=g\circ \sigma: W' \to Y$. Then $g'^*L \sim_{\Q} 0$. Applying Theorem \ref{ct} to $g': W' \to Y$ shows that $L \sim_{\Q} 0$, which is equivalent to that $L_1 \sim_{\mathbb{Q}} L_2$.
\end{proof}

\subsection{Minimal model theory of 3-folds}
The following theorem includes some recent results of minimal model theory for 3-folds in positive characteristic.
\begin{thm}\label{rel-mmp}
Assume $\mathrm{char}~k =p >5$. Let $(X,B)$ be a $\mathbb{Q}$-factorial projective pair of dimension three and $f: X\rightarrow  Y$ a projective surjective morphism.

(1) If either $(X, B)$ is klt and $K_X+B$ is pseudo-effective over $Y$, or $(X, B)$ is lc and $K_X+B$ has a weak Zariski decomposition over $Y$, then $(X,B)$ has a log minimal model over $Y$.

(2) If $(X, B)$ is dlt and $K_X+B$ is not pseudo-effective over $Y$, then $(X,B)$ has a Mori fibre space over $Y$.

(3) Assume that $(X, B)$ is klt in (3.1) and is dlt in other cases, and that $K_X+B$ is nef over $Y$.
\begin{itemize}
\item[(3.1)]
If $K_X+B$ or $B$ is big over $Y$, then $K_X+B$ is semi-ample over $Y$.
\item[(3.2)]
If $\dim Y \geq 1$, then $(K_X+B)_{\eta}$ is semi-ample on $X_{\eta}$.
\item[(3.3)]
If $Y$ is a smooth curve, $X_{\eta}$ is integral and $\kappa(X_{\eta}, (K_X+B)_{\eta}) = 0$ or $2$, then $K_X+B$ is semi-ample over $Y$.
\item[(3.4)]
If $Y$ contains no rational curves, then $K_X+B$ is nef.
\end{itemize}

(4) Assume $(X,B)$ is klt. If $Y$ is a non-uniruled surface and $K_X+B$ is pseudo-effective over $Y$, then $K_X + B$ is pseudo-effective, and there exists a map $\sigma: X \dashrightarrow \bar{X}$ to a minimal model $(\bar{X}, \bar{B})$ of $(X, B)$ such that, the restriction map $\sigma|_{X_{\eta}}$ is an isomorphism from $X_{\eta}$ to its image.
\end{thm}
\begin{proof}
Assertion (1) is from \cite[Theorem 1.2 and Proposition 7.3]{Bir16}; (2) is \cite[Theorem 1.7]{BW14}; (3.1) is \cite[Theorem 1.4]{Bir16}; (3.2) is from  \cite[Theorem 1.1]{Ta15b}; and (3.3) is from \cite[Theorem 1.5 and 1.6 and the remark below 1.6]{BCZ15}.

Assertion (3.4) follows from the cone theorem \cite[Theorem 1.1]{BW14}. Indeed, otherwise we can find an extremal ray $R$ generated by a rational curve $\Gamma$,  so $\Gamma$ is contained in a fiber of $f$ since $Y$ contains no rational curves, this contradicts that $K_X + B$ is $f$-nef.

For (4), $K_X + B$ is obviously pseudo-effective because otherwise, $X$ will be ruled by horizontal (w.r.t. $f$) rational curves by (2), which contradicts that $Y$ is non-uniruled. The exceptional locus of a flip contraction is of dimension one, so it does not intersect $X_{\eta}$, neither does that of an extremal divisorial contraction because it is uniruled (see the proof of \cite[Lemma 3.2]{BW14}). Running an LMMP for $K_X +B$, by induction we get a map $\sigma: X \dashrightarrow \bar{X}$ as required.
\end{proof}

\subsection{Numerical dimension}
\begin{defn}
Let $D$ be an $\R$-Cartier $\R$-divisor on a smooth projective variety $X$ of dimension $n$.
The \emph{numerical dimension} $\kappa_{\sigma}(D)$ is defined as the biggest natural number $k$ such that
$$\liminf_{m \to \infty}\frac{h^0(\llcorner mD \lrcorner + A)}{m^k} >0 \mathrm{~for~some~ample~divisor}~A~\mathrm{on}~X.$$
If such a $k$ does not exist then we define $\kappa_{\sigma}(D) = -\infty$.
\end{defn}

\begin{rem}
In arbitrary characteristics, since smooth alterations exist due to de Jong \cite{J97}, this invariant can be defined for $\R$-Cartier $\R$-divisors on normal varieties by pulling back to a smooth variety, which does not depend on the choices of smooth alterations by \cite[2.5-2.7]{CHMS14}.
\end{rem}

\begin{prop}\label{num-prop}
Let $X$ be a normal projective variety and $D$ an $\R$-Cartier $\R$-divisor on $X$.

(1) When $D$ is nef, then $\kappa_{\sigma}(D)$ coincides with
$$\nu(D) = \max \{k \in \N|D^k\cdot A^{n-k} >0\mathrm{~for~an~ample~divisor}~A~\mathrm{on}~X \}.$$
If moreover $D$ is effective and $S$ is a normal component of $D$, then
$$\nu(D|_S) \leq \nu(D) -1.$$

(2) Let $\mu: W\to X$ be a generically finite surjective morphism between two normal projective varieties and $E$ an effective $\mu$-exceptional $\R$-Cartier $\R$-divisor on $W$. Then
$$\kappa_{\sigma}(D) = \kappa_{\sigma}(\mu^*D + E).$$

(3) If $(X, \Delta)$ is a $\Q$-factorial log canonical 3-fold, and $(X', \Delta')$ is a minimal model of $(X, \Delta)$, then
$$\kappa_{\sigma}(K_X + \Delta) = \nu(K_{X'} + \Delta').$$
\end{prop}
\begin{proof}
For (1), the first assertion is \cite[V, 2.7 (6)]{Nak04} in characteristic zero, and is \cite[Porposition 4.5]{CHMS14} in characteristic $p>0$.
The second assertion follows from the relation \cite[Sec. 1.2]{Deb01}
$$(D|_S)^{k-1}\cdot (A|_S)^{n-k} = D^{k - 1}\cdot A^{n-k} \cdot S \leq D^k\cdot A^{n-k}.$$

(2) is from applying Theorem \ref{ct} (cf. \cite[2.7]{CHMS14}).

Finally for (3), taking a common log resolution of $(X, \Delta)$ and $(X', \Delta')$, this assertion follows from applying (2) and some standard arguments of minimal model theory.
\end{proof}

\subsection{Inseparable fibrations}
In this section we work over an algebraically closed field $k$ of characteristic $p>0$. Let $X$ be a smooth variety. Recall that a (1-)\emph{foliation} is a saturated subsheaf $\mathcal{F} \subset T_X$ which is involutive (i.e., $[\mathcal{F}, \mathcal{F}] \subset \mathcal{F}$) and $p$-closed (i.e., $\xi^p \in \mathcal{F}, \forall \xi \in \mathcal{F}$). A foliation $\mathcal{F}$ is called \emph{smooth} if it is a subbundle of $T_X$.


\begin{prop}\label{flt}
Let $X$ be a smooth variety and $\mathcal{F}$ a foliation on $X$.

(1) We get a normal variety $Y = X/\mathcal{F}= \mathrm{Spec} Ann(\mathcal{F})$, and there exist natural morphisms $\pi:X \to Y$ and $\pi': Y \to X^{(1)}$ fitting into the following commutative diagram
$$
\xymatrix{
&X\ar[d]^{\pi}\ar[dr]^{F_X}  &\\
&Y\ar[r]^>>>>>{\pi'}  &X^{(1)}.
}
$$
Moreover $\deg \pi = p^r$ where $r = \mathrm{rk}~ \mathcal{F}$.

(2) There is a one-to-one correspondence between foliations and normal varieties between $X$ and $X^{(1)}$, by the correspondence $\mathcal{F} \mapsto X/\mathcal{F}$ and the inverse correspondence $Y \mapsto Ann(\O_Y):= \{ \xi \in T_X| \xi(a) = 0, \forall a\in \O_Y\}$.

(3) The variety $Y$ is regular if and only if $\mathcal{F}$ is smooth.

(4) If $Y_0$ denotes the regular locus of $Y$ and $X_0 = \pi^{-1}Y_0$, then
$$K_{X_0} \sim \pi^*K_{Y_0} + (p-1)\det \mathcal{F}|_{X_0}.$$
\end{prop}
\begin{proof}
We refer to \cite[p.56-58]{MP97} or \cite{Ek87}.
\end{proof}

The following result helps us to reduce an inseparable fibration to a separable one.
\begin{prop}\label{red-to-sep}
Let $f: X \to Y$ be a fibration from a normal projective 3-fold to a normal surface of maximal Albanese dimension, and let $B$ be an effective $\Q$-divisor on $X$. Then there exist a purely inseparable morphism
$\sigma: X \to X'$, a separable fibration $f': X'\to Y$,
a rational number $t>0$ and an effective $\Q$-divisor $B'$ on $X'$ such that
$$K_X + B \sim_{\Q} t(\sigma^*(K_{X'} + B')).$$
\end{prop}
\begin{proof}
Let $a_Y: Y \to A$ be the Albanese map and let $a_X = a_Y \circ f$. If $f$ is a separable morphism, then we are done.

Assume that $f$ is inseparable. Let $\mathcal{L}_1$ and $\mathcal{L}_2$ denote the saturation of the image of the natural homomorphisms $a_X^*\Omega_{A_X}^1 \rightarrow \Omega_{X}^1$ and $f^*\Omega_{Y}^1 \rightarrow \Omega_{X}^1$ respectively. Then $\mathcal{L}_1 \subseteq \mathcal{L}_2$ and $\mathrm{rk}~ \mathcal{L}_2 \leq 1$ since $f$ is inseparable (\cite[Prop. 8.6A]{Har77}). And by Igusa's result (\cite[Theorem 4]{Se58a}), we have $\mathcal{L}_1$ is generically globally generated, and $h^0(X, \mathcal{L}_1)\geq h^0(A_X, \Omega_{A_X}^1) \geq 2 > \mathrm{rk}~ \mathcal{L}_1$. Therefore $\mathcal{L}_1 = \mathcal{L}_2 = \mathcal{L}$ is of rank one and $h^0(X, \mathcal{L}^{**}) \geq 2$ (\cite[Lemma 4.2]{Zh16c}).

We get a natural foliation $\mathcal{F} = \mathcal{L}^{\bot} \subset T_X$ of rank $2$, and a quotient map $\rho: X \rightarrow X_1 = X/\mathcal{F}$, which is a factor of $f$ by the construction above. Denote by $X^0$ the maximal smooth open subset of $X$ such that $\mathcal{F}|_{X_0}$ is smooth, and let $X_{1}^0 = X^0/(\mathcal{F}|_{X^0})$.
Then
\begin{equation}\label{can-ins-c}
K_{X^0} \sim \rho^*K_{X_{1}^0} + (p-1)\det \mathcal{F}|_{X^0}.
\end{equation}
On the other hand, we have the following exact sequence
$$0 \rightarrow  \mathcal{L}|_{X^0} \rightarrow \Omega_{X^0}^1 \rightarrow \mathcal{F}^*|_{X^0} \rightarrow 0,$$
which gives
$$\det\mathcal{F}|_{X^0} \sim \mathcal{L}|_{X^0} - K_{X^0}.$$
Combining with Equation (\ref{can-ins-c}), we get
\begin{equation}\label{can-ins-2}
K_{X^0} \sim_{\mathbb{Q}} \frac{1}{p}(\rho^*K_{X_{1}^0} + (p-1) \mathcal{L}|_{X^0}).\end{equation}
Since $\mathcal{L}$ is generically globally generated, there exists an effective Weil divisor $B'$ on $X$ such that $B' \sim \mathcal{L}$. And since $\rho$ is purely inseparable, there exist $\Q$-divisors $B_1, B_1'$ on $X_1$ such that $\rho^*B_1 = B$ and $\rho^*B_1' = B'$. Let $B_1 = pB_1+ (p-1)B_1'$. Then since $X \setminus X^0$ is of codimension $\geq 2$ in $X$, by Equation (\ref{can-ins-2}) we have that
$$K_X+ B \sim_{\Q} \frac{1}{p}(\rho^*(K_{X_1} + B_1)).$$
If the natural fibration $f_1: X_1 \to Y$ is separable then we are done. If not, we consider the pair $(X_1, B_1)$ instead and repeat the process above. We can prove that $\mathrm{mult} ((X_{1})_{\ol\eta}) < \mathrm{mult} (X_{\ol\eta})$ by the argument of \cite[the latter case of Thm. 4.3]{Zh16c}. So this process will terminate, and we can show the assertion by induction.
\end{proof}

For a fibration fibred by curves of arithmetic genus one, we have the following result if $\mathrm{char}~k = p\geq 5$.
\begin{prop}\label{fib-pa=1}
Assume $\mathrm{char}~k = p\geq 5$. Let $g: X \to Z$ be a fibration of normal varieties of relative dimension one. Assume that the generic fiber $X_{\xi}$ of $g$ is a curve with arithmetic genus $p_a(X_{\xi}) = 1$. Then the geometric generic fiber $X_{\ol\xi}$ of $g$ is a smooth elliptic curve over $\ol{K(Z)}$.
\end{prop}
\begin{proof}
Assume by contrary that $X_{\ol\xi}$ is singular. Applying \cite[Lemma 2.3 and 2.4]{Ta15a}, we get the following commutative diagram
$$\xymatrix{
&X \ar[d]^{g} &X_1\ar[l]\ar[d] &X_2 \ar[d]\ar[l] &\tilde{X}_2 \ar[ld]\ar[l] \\
&Z        &Z_1 \ar[l]      &Z_2\ar[l]               &
}$$
where $Z_1 \to Z$ is a purely inseparable base change (or identity), $Z_2 \to Z_1$  is a degree $p$ purely inseparable extension, $X_1 = X\times_Z Z_1$ and $X_2 = X\times_Z Z_2$, such that $X_{\xi_1}$ is regular and that $X_{\xi_2}$ is reduced but not normal, here $\xi_1,\xi_2$ denote the generic point of $Z_1, Z_2$ respectively. The normalization $\tilde{X}_{\xi_2}$ of $X_{\xi_2}$ has smaller arithmetic genus hence must be a smooth curve of genus zero. If necessary by shrinking $Z_i$, we can assume both $\tilde{X}_2$ and $X_1$ are smooth, so the natural morphism $\pi: \tilde{X}_2 \to X_1$ is the quotient induced by a smooth foliation $\mathcal{F}$ on $\tilde{X}_2$, which is a subbundle of $T_{\tilde{X}_2}$ of rank one.
Since
$$K_{\tilde{X}_2} = \pi^*K_{X_1} + (p-1) \mathcal{F}$$
we get a contradiction by $(p-1) \deg \mathcal{F}_{\xi} =  \deg K_{\tilde{X}_{\xi_2}} = -2$ and $p\geq 5$.
\end{proof}

\subsection{Trace maps of absolute Frobenius morphisms}\label{tmaF}
Let $X$ be a projective variety over an algebraically closed field $k$ of characteristic $p>0$. We will consider the trace maps in the following two settings.

\begin{notation}\label{normal}
Assume $X$ is normal.
Denote by $X_0$ the smooth open subset of $X$.
Let $B$ be an effective $\mathbb{Q}$-Weil divisor with Weil index not divisible by $p$. There exists a positive integer $g$ such that $(p^g-1)B$ is integral, thus $(p^{eg}-1)B$ is integral for every integer $e>0$.
The composition map of the natural inclusion
$$F^{eg}_{X*} \mathcal{O}_X ((1-p^{eg})(K_X + B))|_{X_0} \hookrightarrow F^{eg}_{X_0*} \mathcal{O}_{X_0} ((1-p^{eg})K_{X_0})$$
and the trace map $Tr_{F_{X_0}^{eg}}: F^{eg}_{X_0*} \mathcal{O}_{X_0} ((1-p^{eg})K_{X_0}) \rightarrow  \mathcal{O}_{X_0}$ extends to a map on $X$:
$$Tr_{X,B}^{eg}: F^{eg}_{X*} \mathcal{O}_X ((1-p^{eg})(K_X + B)) \rightarrow  \mathcal{O}_X.$$
Let $D$ be a Cartier divisor on $X$. Twisting $Tr_{X,B}^{eg}$ with $\mathcal{O}_X(D)$ induces
\begin{equation*}
\small\begin{split}
Tr_{X,B}^{eg}(D): &F^{eg}_{X*} \mathcal{O}_X ((1-p^{eg})(K_X + B)) \otimes \mathcal{O}_X(D) \\
                     &\cong F^{eg}_{X*} \mathcal{O}_X ((1-p^{eg})(K_X + B)+ p^{eg}D)  \rightarrow  \mathcal{O}_X(D).
\end{split}
\end{equation*}
Then taking global sections induces a trace map
$$\Phi_{eg}: H^0(X, F^{eg}_{X*} \mathcal{O}_X ((1-p^{eg})(K_X + \Delta)+ p^{eg}D)) \rightarrow  H^0(X, D).$$
Denote $S_{B}^{eg}(X, D) = \mathrm{Im} \Phi_{eg}$ and $S_{B}^{0}(X, D) = \cap_{e\geq 0} S_{B}^{eg}(X, D)$. If $B = 0$, we usually use the notation $S^{0}(X, D)$ instead of $S_{0}^{0}(X, D)$. Note that for $e' > e$, we have $S_{B}^{e'g}(X, D) \subseteq S_{B}^{eg}(X, D)$ by the factorization
\begin{equation*}
\small\begin{split}
Tr_{X,B}^{e'g}(D):
&F^{eg}_{X*} F^{(e'-e)g}_{X*}\mathcal{O}_X ((1-p^{e'g})(K_X + B)+ p^{e'g}D) \xrightarrow{F^{eg}_{X*}Tr_{X,\Delta}^{(e'-e)g}((1-p^{eg})(K_X + B)+ p^{eg}D)}\\
& F^{eg}_{X*} \mathcal{O}_X ((1-p^{eg})(K_X + B)+ p^{eg}D) \xrightarrow{Tr_{X,B}^{eg}(D)} \mathcal{O}_X(D).
\end{split}
\end{equation*}
\end{notation}


\begin{notation}\label{S2}
Assume $X$ is Gorenstein in codimension one and satisfies Serre condition $S_2$. Let $B$ be an effective $\mathbb{Q}$-AC divisor such that $K_X + B$ is $\mathbb{Q}$-Cartier (\cite[Sec. 2.1 and 2.3]{Zh16a}, namely, $B= \frac{M - nK_X}{n}$ for some $n>0$, where $M$ is a Cartier divisor and $M - nK_X$ is effective in codimension one). Assume moreover that  the Cartier index of $K_X + B$ is not divisible by $p$. Let $g>0$ be an integer such that $(1-p^{g})(K_X + B)$ is Cartier. Then we can define trace maps $Tr_{X,B}^{eg}, Tr_{X,B}^{eg}(D)$ as in \ref{normal} (see \cite[Sec. 2.3]{Zh16a} for details).
By \cite[Lemma 13.1]{Ga}, there exists an ideal $\sigma(B)$, namely, the \emph{non-$F$-pure ideal} of $(X, B)$, such that for some sufficiently divisible $g'>0$ and any $e>0$,
$$\mathrm{Im} Tr_{X,B}^{eg'} = \sigma(B) = Tr_{X,B}^{eg'} (F^{eg'}_{X*} (\sigma(B)\cdot \mathcal{O}_X ((1-p^{eg'})(K_X + B)))).$$
\end{notation}

Borrowing the idea of the proof of \cite[Lemma 2.20]{Pa14}, we prove the following lemma.
\begin{lem}\label{stable-sec-lem}
Using Notation \ref{S2}, let $A, D$ be two Cartier divisors on $X$. If $A$ is ample, then there exists $M>0$ such that for any $m > M$ and $e>0$, the trace map
\begin{equation*}
\small\begin{split}
\Phi_{e,m}: H^0(X, F_{X*}^{eg'}(\sigma(B)\cdot\O_X((1-p^{eg'})(K_X + B)))&\otimes \O_X(mA+ D))\\
 &\to H^0(X, \sigma(B)\cdot \O_X(mA+D))
 \end{split}
\end{equation*}
is surjective. In particular, there exists some $c>0$ such that for any $m>0$, $\dim S_{B}^{0}(X, mA+D) \geq cm^{\dim X}$.
\end{lem}
\begin{proof} We only need to prove the first assertion, which implies the second one.
Let $K_{eg'}= \mathrm{ker} (Tr_{X,B}^{eg'}: F_{X*}^{eg'}(\sigma(B)\cdot\O_X((1-p^{eg'})(K_X + B))) \to \sigma(B))$. Then we have the following commutative diagram of short exact sequences
{\small $$
\xymatrix{
&0\ar[r]   &K_{eg'}\ar[r]\ar[d]^{\gamma_1}   &F_{X*}^{eg'}(\sigma(B)\cdot\O_X((1-p^{eg'})(K_X + B)))\ar[r]\ar[d]^{\gamma_2} &\sigma(B)\ar@{=}[d]\ar[r] &0\\
&0\ar[r]   &K_{(e-1)g'}\ar[r]   &F_{X*}^{(e-1)g'}(\sigma(B)\cdot\O_X((1-p^{(e-1)g'})(K_X + B)))\ar[r] &\sigma(B)\ar[r] &0\\
}
$$}
where $\gamma_2$ is obtained by applying  $F_{X*}^{(e-1)g'}$ to the trace map
\begin{equation*}
\small\begin{split}
&F_{X*}^{g'}(\sigma(B)\cdot\O_X((1-p^{g'})(K_X + B))) \otimes \O_X((1-p^{(e-1)g'})(K_X + B)) \\
&\cong F_{X*}^{g'}(\sigma(B)\cdot\O_X((1-p^{eg'})(K_X + B)))  \to \sigma(B)\cdot \O_X((1-p^{(e-1)g'})(K_X + B)),
\end{split}
\end{equation*}
and $\gamma_1$ arises naturally. Since $F_{X}^{(e-1)g'}$ is an affine morphism, $\gamma_2$ is surjective and
$$\mathrm{ker}(\gamma_2) = F^{(e - 1)g'}_{X*}(K_{g'}\otimes \mathcal{O}_X ((1-p^{(e - 1)g'})(K_X + B))).$$ Let $K' = \mathrm{ker}(\gamma_1)$. Since $\gamma_2$ is surjective,
applying Snake Lemma we obtain that $\gamma_1$ is surjective and $K'\cong F^{(e - 1)g'}_{X*}(K_{g'}\otimes \mathcal{O}_X ((1-p^{(e - 1)g'})(K_X + B)))$. It follows an exact sequence
\begin{align}\label{exact}
0\to F^{(e - 1)g'}_{X*}(K_{g'}\otimes \mathcal{O}_X ((1-p^{(e - 1)g'})(K_X + B))) \to  K_{eg'} \to  K_{(e - 1)g'} \to 0
\end{align}
And since $A$ is ample, we have
\begin{itemize}
\item[(a)]
there exists $M_0>0$ such that for any $t \in [0,1]$, the divisor $M_0A + D - t(K_X + B)$ is nef; and
\item[(b)]
applying Fujita vanishing (Lemma \ref{rel-fujita}), there exists an integer $M_1$ such that for any $l>M_1$ and any nef Cartier divisor $P$, $H^1(X, K_{g'} \otimes \O_X(lA + P)) = 0$.
\end{itemize}

Let $M = M_0 + M_1$. Fix an integer $m >M$. We aim to show that the trace map $\Phi_{e,m}$ is surjective.
Tensoring the following exact sequence with $\O_X(mA+D)$
$$0\to K_{eg'}\to   F_{X*}^{eg'}(\sigma(B)\cdot\O_X((1-p^{eg'})(K_X + B)))\to \sigma(B)\to 0,$$
and taking cohomology, it is sufficient to show that for every $e>0$,
$$H^1(X, K_{eg'} \otimes \O_X(mA+D)) = 0.$$
By $mA + D = (m-M_0)A + (M_0A + D)$, applying (a) and (b) we can show the case $e=1$. Assume by induction that $H^1(X, K_{(e-1)g'} \otimes \O_X(mA+D)) = 0$ for some $e\geq 2$.
The condition (a) implies that $p^{(e - 1)g'}(M_0A+D) + (1-p^{(e - 1)g'})(K_X + B)$ is nef. Then applying the condition (b), it follows that
\begin{equation*}
\small\begin{split}
&H^1(X, F^{(e - 1)g'}_{X*}(K_{g'}\otimes \mathcal{O}_X ((1-p^{(e - 1)g'})(K_X + B))) \otimes \O_X(mA+D))\\
&\cong H^1(X, K_{g'}\otimes \mathcal{O}_X (p^{(e - 1)g'}(m-M_0)A + p^{(e - 1)g'}(M_0A+D) + (1-p^{(e - 1)g'})(K_X + B)))=0
\end{split}
\end{equation*}
Tensoring the exact sequence (\ref{exact}) with $\O_X(mA+D)$ and taking cohomology, we deduce that $H^1(X, K_{eg'} \otimes \O_X(mA+D)) = 0$.
\end{proof}

\subsection{Derived categories}
Let $X$ be a projective variety, we denote by $D^b(X)$ the
bounded derived category of coherent sheaves on $X$.

For an object $\mathcal{F} \in D^b(X)$ represented by the complex
$$\mathcal{K}^{\bullet}:~ \cdots \to K^{n-1} \to K^{n} \to K^{n+1} \to \cdots$$
we have truncations (\cite[p. 69]{Har66})
$$\sigma_{\leq n}(\mathcal{K}^{\bullet}):~  \cdots \to K^{n-1} \to \mathrm{ker}~d^n \to 0 \to \cdots ~
\mathrm{with}~\mathcal{H}^i(\sigma_{\leq n}(\mathcal{K}^{\bullet})) \cong \mathcal{H}^i(\mathcal{K}^{\bullet})~\mathrm{for}~i\leq n$$
and
$$\sigma_{> n}(\mathcal{K}^{\bullet}):~  \cdots \to 0 \to \mathrm{im}~d^n \to K^{n+1} \to \cdots
\mathrm{with}~\mathcal{H}^i(\sigma_{> n}(\mathcal{K}^{\bullet})) \cong \mathcal{H}^i(\mathcal{K}^{\bullet})~\mathrm{for}~i> n.$$
The exact sequence below
$$0 \to \sigma_{\leq n}(\mathcal{K}^{\bullet}) \to \mathcal{K}^{\bullet} \to \sigma_{> n}(\mathcal{K}^{\bullet}) \to 0,$$
descends to a triangle in $D^b(X)$ (\cite[p. 63 Remark after Prop. 6.1]{Har66})
$$\sigma_{\leq n}(\mathcal{F}) \to \mathcal{F} \to \sigma_{> n}(\mathcal{F}) \to \sigma_{\leq n}(\mathcal{F})[1].$$

Let $f: X \rightarrow Y$ be a projective morphism of projective varieties. Assume $Y$ is smooth. We have derived functors
$Rf_*: D^b(X) \to D^b(Y)$ and $Lf^*: D^b(Y) \to D^b(X)$ of $f_*$ and $f^*$ respectively (\cite[Chap. II Sec. 2,4]{Har66}.
By Grothendieck duality (\cite[Chap. III Sec. 11]{Har66}) there exists a functor $f^!$ such that
$$R\mathcal{H}om_Y(Rf_*\mathcal{E}, \mathcal{F}) \cong Rf_*R\mathcal{H}om_X(\mathcal{E}, f^!\mathcal{F})$$
where $\mathcal{E} \in D^b(X), \mathcal{F}  \in D^b(Y)$.
In particular if both $X$ and $Y$ are smooth, then $f^!\mathcal{F} \cong Lf^*\mathcal{F} \otimes \O_X(K_{X/Y})[\dim X/Y]$ (\cite[Chap. VI Sec. 4]{Har66}).

\subsection{Cohomology of flat complexes under base changes}
Let's recall the following result, which is an adaption of \cite[Chap. III Cor. 12.11]{Har77} to flat bounded complexes.
Though this is known to experts (\cite[Remark of 3.6]{PP11a}), we explain the modifications of the proof for the convenience of the reader.
\begin{thm}\label{coh-res-fib}
Let $f: X \to Y$ be a projective morphism of varieties. Let $\mathcal{K}^{\bullet}$ be a bounded complex of coherent sheaves on $X$ such that every $\mathcal{K}^{i}$ is flat over $Y$. For a closed point $y \in Y$,

(1) if the natural map $\varphi^i_y: R^if_*\mathcal{K}^{\bullet}\otimes k(y) \to R^i\Gamma(X_y, \mathcal{K}^{\bullet}_y)$ is surjective, then $\varphi_y^i$ is an isomorphism, and there exists a neighborhood $U$ of $y$ such that for any $y'\in U$, the map $\varphi^i_{y'}$ is surjective;

(2) if $\varphi^i_y$ is surjective then the following two conditions are equivalent to each other
\begin{itemize}
\item[(2.1)]
$\varphi^{i-1}_y: R^{i-1}f_*\mathcal{K}^{\bullet}\otimes k(y) \to R^{i-1}\Gamma(X_y, \mathcal{K}^{\bullet}_y)$ is surjective;
\item[(2.2)]
$R^if_*\mathcal{K}^{\bullet}$ is locally free at $y$.
\end{itemize}
\end{thm}
\begin{proof}
We assume $Y = \mathrm{Spec}~ A$ is affine. For an $A$-module $M$, define the functor
$$T^i(M) = R^i\Gamma(X, \mathcal{K}^{\bullet} \otimes_A M).$$
To adapt the arguments of \cite[Chap. III, Sec. 12]{Har77} to flat complexes , we only need to verify
the analogue of \cite[Chap. III, Prop. 12.1 and 12.2]{Har77}.

\medskip

(a) Tensoring $\mathcal{K}^{\bullet}$ with a short exact exact sequence $0 \to M' \to M \to M'' \to 0$ of $A$-modules, since $\mathcal{K}^{\bullet}$ is flat, we get a short exact sequence of complexes of sheaves on $X$
$$0 \to \mathcal{K}^{\bullet}\otimes_A M' \to \mathcal{K}^{\bullet}\otimes_A M \to \mathcal{K}^{\bullet}\otimes_A M'' \to 0.$$
Taking cohomology shows that $T^i$ is exact in the middle. So the analogue of \cite[Chap. III, Prop. 12.1]{Har77} holds.

\medskip

(b) Fix an open affine cover $\mathcal{U} = \{U_i\}_{i \in I}$ of $X$. Consider the double complex
$$C^{\bullet}(\mathcal{U}, \mathcal{K}^{\bullet}) = (C^{q}(\mathcal{U}, \mathcal{K}^{p}), d'^{pq}, d''^{pq})$$
where $d'^{pq}$ come from \u{C}ech complexes $C^{\bullet}(\mathcal{U}, \mathcal{K}^p)$ and $d''^{pq}$ are induced by the differentials of $\mathcal{K}^{\bullet}$.
Consider the total complex $C^{\bullet}$
$$C^i = \bigoplus_{p+q = i} C^q(\mathcal{U}, \mathcal{K}^p), d^i =\sum_{p+q = i}(-1)^pd'^{pq} + d''^{pq}:  C^i\to C^{i+1}.$$
Note that $C^i$ is a flat $A$-module, the total complex of $C^{\bullet}(\mathcal{U}, \mathcal{K}^{\bullet} \otimes_A M)$ coincides with $C^{\bullet}\otimes_A M$, thus $H^i(C^{\bullet}\otimes_A M) \cong R^i\Gamma(X, \mathcal{K}^{\bullet} \otimes_A M)$. Then we can show the analogue of \cite[Chap. III, Prop. 12.2]{Har77} by similar arguments.
\end{proof}

\subsection{Abelian subvarieties generated by subschemes}
Let $A$ be an abelian variety of dimension $d$ and $V$ a closed subscheme of $A$. Let $W_n$ be the reduced subscheme supported on the $n$-fold sum $V + V + \cdots + V$. Let $N$ be an integer (say, $N = \dim A$) such that $d' = \dim W_N$ attains the maximum of $\dim W_n$. Let $Z$ be an irreducible component of $W_N$ with $\dim Z = d'$.

\begin{lem}\label{ab-subvar}
With the above notation, take a closed point $z \in Z$ and let $Z_0 = Z - z$. Then

(1) $Z_0$ is an abelian subvariety of $A$; and

(2) every component of $W_n$ is contained in a translate of $Z_0$.
\end{lem}
\begin{proof}
(1) Note that $Z_0 + Z_0$ is irreducible because it is the image of $m: Z_0 \times Z_0 \to A$ via $m(z_1, z_2) = z_1 + z_1$. Since $0 \in Z_0$, we have $\mathrm{Supp} Z_0 \subseteq \mathrm{Supp}(Z_0 + Z_0)$, then $\dim Z_0 = \dim (Z_0 + Z_0)$ implies that $\mathrm{Supp} Z_0 = \mathrm{Supp}(Z_0 + Z_0)$. Applying \cite[ p. 44, Theorem of Appendix to Sec. 4]{Mum74}, we see that $Z_0$ is an abelian subvariety of $A$.

(2) For a component $Z'$, take $z' \in Z'$ and let $Z_0' = Z' - z'$. Since $0\in Z_0'$, $Z_0' + Z_0$ is irreducible and contains both $Z_0'$ and $Z_0$. And since $\dim Z_0$ attains maximum, $Z_0' + Z_0 = Z_0$, which shows $Z_0' \subseteq Z_0$. Then (2) follows easily.
\end{proof}

\section{Sheaves on abelian varieties}\label{sec-shf-ab-var}
In this section, we will study sheaves on abelian varieties. The main result is Theorem \ref{CGG}, which, as a generalization of \cite[Proposition 2.7]{EZ16}, is used in the present paper to study fibrations over abelian varieties. The idea of the proof is stimulated by \cite[Theorem 3.1.1]{HP16}.

\begin{notation}\label{not}
We work over an algebraically closed field $k$. Let $A$ be an abelian variety of dimension $d$, $\hat{A}= \mathrm{Pic}^0(A)$ and $\mathcal{P}$ the Poincar\'{e} line bundle on
$A \times \hat{A}$. Let $p,q$ denote the projections from $A \times \hat{A}$ to $A, \hat{A}$ respectively. The \emph{Fourier-Mukai transform} $R\Phi_{\mathcal{P}}: D^b(A) \rightarrow D^b(\hat{A})$ w.r.t. $\mathcal{P}$ is defined as
$$R\Phi_{\mathcal{P}}(-) := Rq_*(Lp^*(-)\otimes \mathcal{P})$$
which is a right derived functor.
Similarly $R\Psi_{\mathcal{P}}: D^b(\hat{A}) \rightarrow D^b(A)$ is defined as
$$R\Psi_{\mathcal{P}}(-) := Rp_*(Lq^*(-) \otimes \mathcal{P}).$$

If $L$ is an ample line bundle on $\hat{A}$, then $H^i(A, L\otimes \mathcal{P}_t) = 0$ for $i>0$ and every $t \in A$ (\cite[Sec. 13]{Mum74}), thus $\hat{L}: = R^0\Psi_{\mathcal{P}}L \cong R\Psi_{\mathcal{P}}L$ is a locally free sheaf of rank $h^0(\hat{A}, L)$ by Theorem \ref{coh-res-fib}.

Note that since $p, q$ are smooth morphisms, we have $Lp^* \cong p^*, Lq^*\cong q^*$. In what follows, for an isogeny $\pi: A_1 \to A$ of abelian varieties, by abuse of the notation of the pull-back map of sheaves, we use $\pi^*: \hat{A} \to \hat{A}_1$ for the dual map of $\pi$.

For a coherent sheaf $\mathcal{F}$ on $A$, we define
$$D_A(\mathcal{F}) := R\mathcal{H}om_X(\mathcal{F}, \omega_A)[d],$$
then applying Grothendieck duality we have
$$D_k(R\Gamma(\mathcal{F})) \cong R\Gamma(D_A(\mathcal{F})).$$

For a closed point $t_0 \in A$, the translating morphism $T_{t_0}: A \to A$ is defined via $t \mapsto t+ t_0$. For $\hat{t}_0 \in \hat{A}$, $T_{\hat{t}_0}$ is similarly defined.
\end{notation}

\begin{thm}[\cite{Mu81}]\label{Mu} Using Notation \ref{not}, we have
\begin{itemize}
\item[(1)]
$R\Psi_{\mathcal{P}} \circ R\Phi_{\mathcal{P}} \cong (-1)_A^*[-d],~~~R\Phi_{\mathcal{P}} \circ R\Psi_{\mathcal{P}} \cong (-1)_{\hat{A}}^*[-d]$;
\item[(2)]
$R\Phi_{\mathcal{P}}\circ (-1)_A^* \cong (-1)_{\hat{A}}^* \circ R\Phi_{\mathcal{P}}, R\Psi_{\mathcal{P}}\circ (-1)_{\hat{A}}^* \cong (-1)_{A}^* \circ R\Psi_{\mathcal{P}}$; and
\item[(3)]
for $\hat{t}_0 \in \hat{A}$, $R\Psi_{\mathcal{P}} \circ T_{\hat{t}_0}^* \cong \mathcal{P}_{-\hat{t}_0} \otimes R\Psi_{\mathcal{P}}$.
\end{itemize}
\end{thm}

\begin{defn}\textup{(\cite[Def. 3.1]{PP11a})}\label{defgv}
Given a coherent sheaf $\mathcal{F}$ on an abelian variety $A$, its $i^{\mathrm{th}}$ \emph{cohomological support locus} is defined as
$$V^i(\mathcal{F}): = \{\alpha \in \mathrm{Pic}^0(A)| h^i(\mathcal{F} \otimes \alpha) > 0\}$$
which are Zariski closed by semi-continuity.
If
$$gv(\mathcal{F}): = \mathrm{min}_{i>0}\{\mathrm{codim}_{\mathrm{Pic}^0(A)}V^i(\mathcal{F}) - i\} \geq 0$$
we say $\mathcal{F}$ is a \emph{GV-sheaf}.
\end{defn}

\begin{prop}\label{fm-gv-sh} Using Notation \ref{not}, let $\mathcal{F}$ be a coherent sheaf on $A$.

(1) For an ample line bundle $H$ on $\hat{A}$, there is a natural isomorphism
$$D_k(R\Gamma(\mathcal{F}\otimes \hat{H}^*))
\cong R\Gamma(R\Phi_{\mathcal{P}}D_A(\mathcal{F})\otimes H),$$
in particular, $H^i(A, \mathcal{F}\otimes \hat{H}^*)^* \cong R^{-i}\Gamma(R\Phi_{\mathcal{P}}D_A(\mathcal{F})\otimes H)$ for every integer $i$.

(2) The Fourier-Mukai transform
$$R\Phi_{\mathcal{P}}D_A(\mathcal{F}) \in D^{[-d, 0]}(A)~\mathrm{and}~\mathrm{Supp} (-1)_{\hat{A}}^*R^0\Phi_{\mathcal{P}}D_A(\mathcal{F}) = V^0(\mathcal{F}).$$

(3) The following three conditions are equivalent to each other
\begin{itemize}
\item[(i)]
the sheaf $\mathcal{F}$ is a GV-sheaf;
\item[(ii)]
$R\Phi_{\mathcal{P}} D_A(\mathcal{F}) \cong R^0\Phi_{\mathcal{P}} D_A(\mathcal{F})$;
\item[(iii)]
for any sufficiently ample line bundle $L$ on $\hat{A}$,
$$H^i(A, \mathcal{F} \otimes \hat{L}^*) = 0~ \mathrm{for} ~ i>0.$$
\end{itemize}
\end{prop}
\begin{proof}
(1) is contained in the proof of \cite[Theorem 1.2]{Hac04}, which follows from applying Grothendieck duality and projection formula.

\medskip

(2) Take an ample line bundle $H$ on $\hat{A}$. By (1) we have that
\begin{itemize}
\item[(a)]
$R^{-i}\Gamma(R\Phi_{\mathcal{P}}D_A(\mathcal{F})\otimes H) = 0$ unless $0\leq i \leq d$.
\end{itemize}
Since $R\Phi_{\mathcal{P}}D_A(\mathcal{F}) \in D^b(\hat{A})$, we can assume $H$ is sufficiently ample such that, for every $q$,
\begin{itemize}
\item[(b)]
$R^q\Phi_{\mathcal{P}}D_A(\mathcal{F})\otimes H$ is globally generated, and
\item[(c)]
$R^{p}\Gamma(R^q\Phi_{\mathcal{P}}D_A(\mathcal{F})\otimes H) = 0$ whenever $p \neq 0$.
\end{itemize}
Since $H$ is a line bundle, we have the spectral sequence
$$E_2^{p,q} := R^{p}\Gamma(R^q\Phi_{\mathcal{P}}D_A(\mathcal{F})\otimes H) \cong R^{p}\Gamma(\mathcal{H}^q(R\Phi_{\mathcal{P}}D_A(\mathcal{F})\otimes^{L} H)) \Rightarrow R^{p+q}\Gamma(R\Phi_{\mathcal{P}}D_A(\mathcal{F})\otimes H).$$
By (c) we can show
$$E_{\infty}^{0,q} = E_2^{0,q} = \Gamma(R^q\Phi_{\mathcal{P}}D_A(\mathcal{F})\otimes H) \cong R^{q}\Gamma(R\Phi_{\mathcal{P}}D_A(\mathcal{F})\otimes H).$$
Then by (a) and (b), we conclude that
$R^i\Phi_{\mathcal{P}}D_A(\mathcal{F}) = 0~ \mathrm{unless}~ -d~\leq i \leq 0$, that is, $R\Phi_{\mathcal{P}}D_A(\mathcal{G}) \in D^{[-d, 0]}(A)$.

For $\alpha \in \hat{A}$, applying Grothendieck duality we have
$$H^j(\mathcal{F}\otimes \mathcal{P}_{\alpha})^* \cong R^{-j}\Gamma(D_A(\mathcal{F})\otimes \mathcal{P}_{-\alpha}).$$
In particular, $R^i\Gamma(D_A(\mathcal{F})\otimes \mathcal{P}_{-\alpha}) = 0$ for $i>0$.
Applying $Rq_*$ to $p^*D_A(\mathcal{F})\otimes \mathcal{P}$ on $A \times \hat{A}$, since $R^1\Phi_{\mathcal{P}}D_A(\mathcal{F}) = 0$, by Theorem \ref{coh-res-fib}, we conclude
$$H^0(\mathcal{F}\otimes \mathcal{P}_{\alpha})^* \cong R^0\Gamma(D_A(\mathcal{F})\otimes \mathcal{P}_{-\alpha}) \cong R^0\Phi_{\mathcal{P}}D_A(\mathcal{F}) \otimes k(-\alpha),$$
thus $\mathrm{Supp} (-1)_{\hat{A}}^*R^0\Phi_{\mathcal{P}}D_A(\mathcal{F}) = V^0(\mathcal{F})$.

\medskip

(3) follows from applying \cite[Theorem 1.2]{Hac04} and \cite[Lemma 3.6]{PP11a}.
\end{proof}

Part of the following theorem is known to experts, in particular assertion (4) also appeared in \cite[Sec. 1.2]{HPZ17}, which states a special reature of positive characteristic.
\begin{thm}\label{pf-OA}
Using Notation \ref{not}, then

(1) If $\tau$ is a coherent sheaf supported at finitely many closed points on $\hat{A}$ of length $r$, then
$U = R\Psi_{\mathcal{P}}\tau = R^0\Psi_{\mathcal{P}}\tau$ is a vector bundle of rank $r$ (homogenous vector bundle); moreover if $\mathrm{Supp} \tau = \{\hat{0}\}$ then $U$ is a unipotent vector bundle, that is, $U$ admits a filtration of vector bundles
$$0 = U_0 \subset U_1 \subset U_2 \subset \cdots \subset U_r = U$$
such that $U_i/U_{i-1} \cong \O_X$.

(2) If $\pi: A_1 \to A$ is an isogeny of abelian varieties, then
$\pi_*\O_{A_1} \cong \bigoplus_i U_{\hat{t}_i}$ where
\begin{itemize}
\item
$\hat{t}_i$ are finitely many closed points on $\hat{A}$ such that $\pi^*\mathcal{P}_{\hat{t}_i} \cong \O_{A_1}$, and
\item
$U_{\hat{t}_i} = R\Psi_{\mathcal{P}}\tau_i$ where $\tau_i$ is a skyscraper sheaf supported at $\hat{t}_i$.
\end{itemize}

(3) Let $\mathcal{F}$ be a coherent sheaf on $A$. Set $V^{\geq m}(\mathcal{F}) = \cup_{j\geq m} V^j(\mathcal{F})$.
If $\pi: A_1 \to A$ is an isogeny of abelian varieties, then
$$V^{m}(\pi^*\mathcal{F}) \subseteq \pi^* V^{m}(\mathcal{F})~ \mathrm{and}~ V^{\geq m}(\pi^*\mathcal{F}) = \pi^* V^{\geq m}(\mathcal{F}).$$
In particular, if for every $j> 0$, $h^j(A_1, \pi^*\mathcal{F}) = 0$, then $h^j(A, \mathcal{F}) = 0$ for $j>0$.

Moreover, we have $V^{0}(\pi^*\mathcal{F}) = \pi^* V^{0}(\mathcal{F})$.

(4) Assume $\mathrm{char}~k = p >0$. Let $\tau$ be a coherent sheaf supported on the union of finitely many closed points on $\hat{A}$ and let $U = R\Psi_{\mathcal{P}}\tau$. Then there exists an isogeny $\mu: B \to A$ of abelian varieties such that $\mu^*U \cong \bigoplus_i \mathcal{P}'_{\hat{s}_i}$ for some closed points $\hat{s}_i$ on $\hat{B}= \mathrm{Pic}^0(B)$, where $\mathcal{P}'$ denotes the Poincar\'{e} line bundle on $B\times \hat{B}$.
\end{thm}
\begin{proof}
(1) Since $\dim \mathrm{Supp} \tau = 0$, for any closed point $t \in A$ we have
$$h^i(\hat{A}, \tau \otimes \mathcal{P}_t) = 0~\mathrm{if}~i\neq 0~\mathrm{and}~h^0(\hat{A}, \tau \otimes \mathcal{P}_t) = r.$$
The first assertion of (1) follows from applying Theorem \ref{coh-res-fib}.

If moreover $\tau$ is a skyscraper sheaf supported at $\hat{0}$, then we have a filtration
$$0 = \tau_0 \subset \tau_1 \subset \tau_2 \subset \cdots \subset \tau_r = \tau$$
such that $\tau_i/\tau_{i-1} \cong k(\hat{0})$, which, by applying Fourier-Mukai transform $R\Psi_{\mathcal{P}}$, induces a filtration of $U$ as wanted.

\medskip

(2) For $\hat{t} \in \hat{A}$, we have $H^i(A_1, \pi^*\mathcal{P}_{\hat{t}}) \cong H^i(A, \pi_*\O_{A_1}\otimes \mathcal{P}_{\hat{t}})$. It follows that
$$V^0(\pi_*\O_{A_1}) = V^1(\pi_*\O_{A_1})  = \cdots = V^d(\pi_*\O_{A_1}) = S_{\pi}:= \{\hat{t}' \in \hat{A}| \pi^*\mathcal{P}_{\hat{t}'} \cong \O_{A_1}\}.$$
We have $S_{\pi} = \mathrm{Supp} (\mathrm{ker}~\pi^*)$, thus $\dim V^k(\pi_*\O_{A_1}) = 0$ for $0\leq k \leq d$, and $\pi_*\O_{A_1}$ is a GV-sheaf. By Grothendieck duality
$$D_A(\pi_*\O_{A_1}) \cong \pi_*D_{A_1}(\O_{A_1}) \cong \pi_*\O_{A_1}[d],$$
applying Proposition \ref{fm-gv-sh} (3) we have
$$R\Phi_{\mathcal{P}} \pi_*\O_{A_1}[d] \cong R^d\Phi_{\mathcal{P}} \pi_*\O_{A_1}~\mathrm{and}~\mathrm{Supp}R^d\Phi_{\mathcal{P}} \pi_*\O_{A_1} = S_{\pi}.$$
We can assume that $R^d\Phi_{\mathcal{P}}\pi_*\O_{A_1} = \bigoplus_{i=1}^{i=n} \tau'_i$
where every $\tau'_i$ is a skyscraper sheaf supported at some $\hat{t}'_i \in S_{\pi}$.
Applying Theorem \ref{Mu}, we have that
$$\pi_*\O_{A_1} \cong (-1)_A^*R^0\Psi_{\mathcal{P}} \bigoplus_{i=1}^{i=n}\tau'_i \cong \bigoplus_{i=1}^{i=n}R^0\Psi_{\mathcal{P}} (-1)_{\hat{A}}^*\tau'_i.$$
So we only need to set $\tau_i = (-1)_{\hat{A}}^*\tau'_i, \hat{t}_i = (-1)_{\hat{A}}^*\hat{t}'_i$ and $U_{\hat{t}_i} = R^0\Psi_{\mathcal{P}}\tau_i$.

\medskip

(3) We use the notation of (2).
For $\hat{t} \in \hat{A}$, we have
$$H^j(A_1, \pi^*(\mathcal{F}\otimes \mathcal{P}_{\hat{t}})) \cong H^j(A, \mathcal{F} \otimes \mathcal{P}_{\hat{t}}\otimes \pi_*\O_{A_1}) \cong H^j(A, \bigoplus_i\mathcal{F}\otimes \mathcal{P}_{\hat{t}} \otimes U_{\hat{t}_i}).$$
Let $\tau_i^0 = T_{\hat{t_i}}^*\tau_i$ and $U_i^0 = R^0\Psi_{\mathcal{P}} \tau_i^0$. Then $\tau_i^0$ is supported at $\hat{0}$, $U_{\hat{t}_i} \cong \mathcal{P}_{\hat{t}_i} \otimes U_i^0$ and
$$H^j(A_1, \pi^*(\mathcal{F}\otimes \mathcal{P}_{\hat{t}})) \cong \bigoplus_i H^j(A,\mathcal{F}\otimes \mathcal{P}_{\hat{t} + \hat{t}_i}\otimes U_i^0).$$

\textbf{Claim}: \emph{For a coherent sheaf $\mathcal{G}$ and a unipotent vector bundle $U$ on $A$,}

(3.1') \emph{if $H^l(A, \mathcal{G}) = 0$ then $H^l(A, \mathcal{G} \otimes U) = 0$; and}

(3.2') \emph{if $H^j(A, \mathcal{G} \otimes U) = 0$ for every $j \geq m$, then $H^j(A, \mathcal{G}) = 0$ for $j\geq m$.}

Granted the claim above, we prove assertion (3) by the following arguments.

(3.1) For $\hat{t} \in \hat{A}$, if $\pi^*\hat{t} \in V^{m}(\pi^*\mathcal{F})$, then there exists some $\hat{t}_i \in S_{\pi}$ such that $H^m(A,\mathcal{F}\otimes \mathcal{P}_{\hat{t} + \hat{t}_i}\otimes U_i^0) \neq 0$. Applying (3.1') shows that $H^m(A,\mathcal{F}\otimes \mathcal{P}_{\hat{t} + \hat{t}_i}) \neq 0$, i.e., $\hat{t} + \hat{t}_i \in V^{m}(\mathcal{F})$. Since $\pi^*\hat{t}_i = \hat{0}$ and $\pi^*: \hat{A} \to \hat{A}_1$ is an epimorphism, we see that
$$V^{m}(\pi^*\mathcal{F}) \subseteq \pi^*V^{m}(\mathcal{F}).$$

(3.2) For $\hat{t} \in \hat{A}$, if $\hat{t} \in V^{\geq m}(\mathcal{F})$ then $H^{j}(A,\mathcal{F}\otimes \mathcal{P}_{\hat{t}}) \neq 0$ for some $j\geq m$. Since $\hat{0} \in S_{\pi}$, $\pi_*\O_{A_1}$ has a unipotent direct summand $U$. Applying (3.2') shows that there exists some $j' \geq m$ such that
$H^{j'}(A,\mathcal{F}\otimes \mathcal{P}_{\hat{t}}\otimes U) \neq 0$, thus $H^{j'}(A_1, \pi^*(\mathcal{F}\otimes \mathcal{P}_{\hat{t}})) \neq 0$. Therefore,
$\pi^*V^{\geq m}(\mathcal{F}) \subseteq V^{\geq m}(\pi^*\mathcal{F})$, and the equality holds by combining with (3.1).

(3.3) To show $V^{0}(\pi^*\mathcal{F}) = \pi^*V^{0}(\mathcal{F})$, by (3.1) it suffices to show that $\pi^*V^{0}(\mathcal{F}) \subseteq V^{0}(\pi^*\mathcal{F})$, which follows from the fact that the natural map $\pi^*: H^0(A, \mathcal{F} \otimes \mathcal{P}_{\hat{t}}) \to H^0(A_1, \pi^*(\mathcal{F} \otimes \mathcal{P}_{\hat{t}}))$ is injective since $\pi$ is flat.

\emph{Proof of Claim.} Let $r = \mathrm{rk} ~U$ and $U_r = U$. By (1) we have a filtration of vector bundles
$$0 = U_0 \subset U_1 \subset U_2 \subset \cdots \subset U_r = U,$$
in turn we get the following short exact sequences
$$\begin{array}{rl}
0 \to \mathcal{G} \otimes U_{r-1} \to \mathcal{G} \otimes U_r \to \mathcal{G} \to 0 ~~&(r)\\
0 \to \mathcal{G} \otimes U_{r-2} \to \mathcal{G} \otimes U_{r-1} \to \mathcal{G} \to 0 ~~&(r-1)\\
\cdots~~&\cdots    \\
0 \to \mathcal{G} \to \mathcal{G} \otimes U_2 \to \mathcal{G} \to 0 ~~&(2).
\end{array}$$

(3.1') Taking cohomology of those short exact sequences (2, 3, $\cdots$, r), since $H^l(A, \mathcal{G}) = 0$ we can show that
$$H^l(A, \mathcal{G}\otimes U_2) = H^l(A, \mathcal{G}\otimes U_3) = \cdots = H^l(A, \mathcal{G}\otimes U_r) = 0.$$

(3.2') We will prove  $H^j(A, \mathcal{G}) = 0$ for $j \geq m$ by induction on $j$. This is trivial if $j > d$.
Assume that we have proved for some fixed $l>m$, $$H^i(A, \mathcal{G}) = 0~\mathrm{whenever}~i \geq l,$$
which, by (3.1'), implies that for $i\geq l$
$$H^i(A, \mathcal{G}\otimes U_2) = H^i(A, \mathcal{G}\otimes U_3) = \cdots = H^i(A, \mathcal{G}\otimes U_r) = 0.$$
Taking cohomology of the short exact sequence (r), by the vanishing $H^{l-1}(A, \mathcal{G}\otimes U_r) = 0$ and $H^{l}(A, \mathcal{G}\otimes U_{r-1}) = 0$, we show that
$$H^{l-1}(A, \mathcal{G}) = 0.$$
By induction, we finish the proof of this claim.
\medskip

(4) We can write that $\tau = \bigoplus_i \tau_i$ where every $\tau_i$ is a sheaf supported at one certain closed point $\hat{t}_i$. Let $U_i = R^0\Psi_{\mathcal{P}}\tau_i$.
Then $U = \bigoplus_i U_i$. We only need to show assertion (4) for a single $U_i$.

We can assume $\tau$ is supported at exactly one closed point $\hat{t} \in \hat{A}$.
Let
$$\bar{\tau} = T_{\hat{t}}^*\tau~\mathrm{and}~\bar{U} = R\Psi_{\mathcal{P}}\bar{\tau}.$$
Then $\mathrm{Supp}\bar{\tau} = \{\hat{0}\}$, $U \cong \bar{U}\otimes \mathcal{P}_{\hat{t}}$, and $\bar{U}$ is a unipotent vector bundle.
We can consider $\bar{U}$ instead and do induction on the rank. By (1) there is a filtration of vector bundles
$$0 = \bar{U}_0 \subset \bar{U}_1 \subset \bar{U}_2 \subset  \cdots \bar{U}_{i-1} \subset \bar{U}_i \subset \cdots \subset \bar{U}_r = \bar{U}.$$
Assume by induction that for some $i\leq r$ there exists an isogeny $\mu_{i-1}: B_{i-1} \to A$ such that $\mu_{i-1}^*\bar{U}_{i-1} \cong \bigoplus^{i-1} \O_{B_{i-1}}$.
Then we have the extension
$$0 \to \bigoplus^{i-1} \O_{B_{i-1}} \to \mu_{i-1}^*\bar{U}_i \to \O_{B_{i-1}} \to 0$$
which corresponds to
$$(\alpha_1, \alpha_2, \cdots, \alpha_{i-1}) \in \bigoplus^{i-1} H^1(\O_{B_{i-1}}) \cong Ext^1(\O_{B_{i-1}}, \bigoplus^{i-1} \O_{B_{i-1}}).$$
Recall ``killing cohomology'', which says that for a projective variety $X$ and $\alpha \in H^1(\mathcal{O}_X)$, there exists a morphism $\pi: Y \to X$ composed with some Frobenius iterations and \'{e}tale $\mathbb{Z}/(p)$-covers such that $\pi^*\alpha = 0$ in $H^1(\mathcal{O}_Y)$ (\cite[Prop. 12 and Sec. 9]{Se58b}). So applying ``killing cohomology'', we get a base change $\nu_{i}: B_i \to B_{i-1}$, which is an isogeny of abelian varieties such that
$$\nu_i^*\alpha_1 = \nu_i^*\alpha_2 = \cdots \nu_i^*\alpha_{i-1} = 0 \in H^1(\O_{B_i}).$$
Let $\mu_i = \mu_{i-1} \circ \nu_i: B_i \to A$. Then
$$\mu_i^*\bar{U}_i \cong \nu_i^*\mu_{i-1}^*\bar{U}_i \cong \bigoplus^i \O_{B_i}.$$
We finish the proof.
\end{proof}

\begin{lem}\label{gen-th-sh} Using Notation \ref{not}, for a coherent sheaf $\mathcal{G}$ on $A$, there exists a natural homomorphism
$$
\alpha_{\mathcal{G}}: \mathcal{G}^* = \mathcal{E}xt^0(\mathcal{G}, \mathcal{O}_A) \to (-1)_A^*R^{0}\Psi_{\mathcal{P}}R^{0}\Phi_{\mathcal{P}}D_A(\mathcal{G})
$$
with the kernel $\mathcal{K}_{\mathcal{G}} \cong (-1)_A^*R^{0}\Psi_{\mathcal{P}} (\sigma_{\leq -1}R\Phi_{\mathcal{P}}D_A(\mathcal{G}))$.
\end{lem}
\begin{proof}
Apply $(-1)_A^*R\Psi_{\mathcal{P}}$ to the following triangle
\begin{align*}
\sigma_{\leq -1}R\Phi_{\mathcal{P}}D_A(\mathcal{G}) \to R\Phi_{\mathcal{P}}D_A(\mathcal{G}) \to \sigma_{> -1}R\Phi_{\mathcal{P}}D_A(\mathcal{G})
\to \sigma_{\leq -1}R\Phi_{\mathcal{P}}D_A(\mathcal{G})[1]
\end{align*}
and take cohomology.
Since $\sigma_{> -1}R\Phi_{\mathcal{P}}D_A(\mathcal{G}) \cong R^0\Phi_{\mathcal{P}}D_A(\mathcal{G})$ (Proposition \ref{fm-gv-sh} (2)), we get the following exact sequence
\begin{align*}
(-1)_A^*R^{-1}\Psi_{\mathcal{P}}R^0\Phi_{\mathcal{P}}D_A(\mathcal{G}) \to &(-1)_A^*R^{0}\Psi_{\mathcal{P}}(\sigma_{\leq -1}R\Phi_{\mathcal{P}}D_A(\mathcal{G})) \\
&\to (-1)_A^*R^{0}\Psi_{\mathcal{P}}R\Phi_{\mathcal{P}}D_A(\mathcal{G}) \xrightarrow{\alpha'_{\mathcal{G}}} (-1)_A^*R^{0}\Psi_{\mathcal{P}}R^0\Phi_{\mathcal{P}}D_A(\mathcal{G}).
\end{align*}
Applying Theorem \ref{Mu}, we have an isomorphism
$$\mathcal{E}xt^0(\mathcal{G}, \mathcal{O}_A) \cong (-1)_A^*\mathcal{H}^{0}(R\Psi_{\mathcal{P}}R\Phi_{\mathcal{P}}D_A(\mathcal{G})),$$
then we get the homomorphism $\alpha_{\mathcal{G}}$ by composing $\alpha'_{\mathcal{G}}$ with this isomorphism.
Since $R^{-1}\Psi_{\mathcal{P}}R^0\Phi_{\mathcal{P}}D_A(\mathcal{G}) =0$, it follows that the kernel of $\alpha_{\mathcal{G}}$
$$\mathcal{K}_{\mathcal{G}} \cong (-1)_A^*R^{0}\Psi_{\mathcal{P}} (\sigma_{\leq -1}R\Phi_{\mathcal{P}}D_A(\mathcal{G})).$$
\end{proof}

Let us prove the main theorem of this section.
\begin{thm}\label{CGG} Using Notation \ref{not}, let $\mathcal{F}= \mathcal{F}_0, \mathcal{F}_1, \cdots, \mathcal{F}_d$ be torsion free coherent sheaves on $A$ equipped with homomorphisms $\phi_e: \mathcal{F}_e \to \mathcal{F}_{e-1}, e=1,2,\cdots, d$. Let $\varphi_e = \phi_e \circ \phi_{e-1}\circ\cdots \circ \phi_1: \mathcal{F}_e \to \mathcal{F}$. And let $H_l, l=0,1,\cdots, d-1$ be ample line bundles on $\hat{A}$. Assume that
\begin{itemize}
\item[(a)]
for $0\leq l \leq d-1$ and every $i$, the sheaf $R^i\Phi_{\mathcal{P}}D_A(\mathcal{F}_l) \otimes H_l$ is globally generated, and if $j>0$ then $H^j(\hat{A}, R^i\Phi_{\mathcal{P}}D_A(\mathcal{F}_l) \otimes H_l) = 0$;
\item[(b)]
for $0\leq l < m \leq d$, if $j>0$ then $H^j(A, \mathcal{F}_m\otimes \hat{H_l}^*) = 0$; and
\item[(c)]
the dual homomorphism $\mathcal{F}^* \to \mathcal{F}_d^*$ of $\varphi_d$ is injective.
\end{itemize}
Then
\begin{itemize}
\item[(i)]
the homomorphism $\alpha_{\mathcal{F}}: \mathcal{F}^{*} \to (-1)_A^*R^{0}\Psi_{\mathcal{P}}R^{0}\Phi_{\mathcal{P}}D_A(\mathcal{F})$ $($introduced in Lemma \ref{gen-th-sh}$)$ is injective; and
\item[(ii)]
if moreover $\mathrm{char}~k = p>0$ and $R^0\Phi_{\mathcal{P}}D_A(\mathcal{F}) = \tau$ is supported at finitely many closed points, then there exist an isogeny $\pi: A_1\to A$ of abelian varieties, some $P_i \in \mathrm{Pic}^0(A_1)$ and a generically surjective homomorphism
$$\beta_{\mathcal{F}}: \bigoplus_iP_i \to \pi^*\mathcal{F}.$$
\end{itemize}
\end{thm}
\begin{proof}
The maps $\varphi_e: \mathcal{F}_e \to \mathcal{F}$ induce $\varphi_e^*: D_A(\mathcal{F}) \to D_A(\mathcal{F}_e)$ by taking dual, then applying $R\Phi_{\mathcal{P}}$ induces natural homomorphisms
$$R\Phi_{\mathcal{P}}D_A(\mathcal{F}) \to R\Phi_{\mathcal{P}}D_A(\mathcal{F}_e).$$

We have the following lemma with the proof postponed.
\begin{lem}\label{zeromap}
The natural homomorphism
$$R^0\Psi_{\mathcal{P}}(\sigma_{\leq -1}R\Phi_{\mathcal{P}}D_A(\mathcal{F})) \to R^0\Psi_{\mathcal{P}}(\sigma_{\leq -1}R\Phi_{\mathcal{P}}D_A(\mathcal{F}_d))$$
is zero.
\end{lem}

\medskip

(i) Applying Lemma \ref{gen-th-sh}, we have the following commutative diagram
$$\xymatrix{
&0 \ar[r] &\mathcal{K}_{\mathcal{F}}\ar[r]\ar[d] &\mathcal{F}^*\ar[r]^>>>>{\alpha_{\mathcal{F}}}\ar[d]  &(-1)_A^*R^{0}\Psi_{\mathcal{P}}R^{0}\Phi_{\mathcal{P}}D_A(\mathcal{F})\ar[d] \\
&0 \ar[r] &\mathcal{K}_{\mathcal{F}_d}\ar[r] &\mathcal{F}_d^*\ar[r]^>>>>{\alpha_{\mathcal{F}_d}} &(-1)_A^*R^{0}\Psi_{\mathcal{P}}R^{0}\Phi_{\mathcal{P}}D_A(\mathcal{F}_d),
}$$
and the vertical map $\mathcal{K}_{\mathcal{F}} \to \mathcal{K}_{\mathcal{F}_d}$ is zero by Lemma \ref{zeromap}.  Then from the condition (c), we conclude
$\mathcal{K}_{\mathcal{F}} = 0$, thus $\alpha_{\mathcal{F}}$ is injective.

\medskip

(ii) Let $U = (-1)_A^*R^0\Psi_{\mathcal{P}}\tau$. Then by Theorem \ref{pf-OA} (1) and (4), the sheaf $U$ is locally free, and there exists an isogeny $\pi: A_1\to A$ of abelian varieties such that
$$\pi^*U \cong \bigoplus_i Q_i~\mathrm{for}~\mathrm{some}~Q_i \in \mathrm{Pic}^0(A_1).$$
Applying $(-1)_A^*R\Psi_{\mathcal{P}}$ to the natural homomorphism $R\Phi_{\mathcal{P}}D_A(\mathcal{F}) \to \tau$ induces a homomorphism
$$\gamma: R\mathcal{H}om(\mathcal{F}, \O_A) \cong D_A(\mathcal{F})[-d] \to U,$$
then applying $R\mathcal{H}om(\cdot, \mathcal{O}_{A})$ induces
$$\gamma^*: U^* \to \mathcal{F}.$$
Since by (i) the homomorphism $\alpha_{\mathcal{F}}: \mathcal{F}^* \to U$ is injective, $\gamma^*$ is generically surjective. It follows that the pull-back homomorphism via $\pi$
$$\beta_{\mathcal{F}}: \pi^*U^* \to \pi^*\mathcal{F}$$
is surjective over the generic point of $A_1$. So we are done by setting $P_i = Q_i^*$.

\medskip

\emph{Proof of Lemma \ref{zeromap}.}  We divide the proof into two steps.
\medskip

Step 1: We prove that
\emph{for fixed $0\leq l < m \leq d$ and any $s\leq -1$, the natural map $R^{s}\Phi_{\mathcal{P}}D_A(\mathcal{F}_l) \to R^{s}\Phi_{\mathcal{P}}D_A(\mathcal{F}_m)$ is zero.}

Consider the spectral sequence
$$E_{2, \mathcal{F}_l}^{r,s}:= R^r\Gamma(R^{s}\Phi_{\mathcal{P}}D_A(\mathcal{F}_l)\otimes H_l) \Rightarrow R^{r+s}\Gamma(R\Phi_{\mathcal{P}}D_A(\mathcal{F}_l)\otimes H_l).$$
By the condition (a) and Proposition \ref{fm-gv-sh} (2), we see that
$$E_{\infty, \mathcal{F}_l}^{r,s} \cong E_{2, \mathcal{F}_l}^{r,s} =0~\mathrm{unless}~r= 0~\mathrm{and}~-d\leq s \leq 0,$$
thus for $-d\leq s \leq 0$, the following natural homomorphism is an isomorphism
$$
\gamma_{\mathcal{F}_l}^s: R^{s}\Gamma(R\Phi_{\mathcal{P}}D_A(\mathcal{F}_l)\otimes H_l) \to E_{\infty, \mathcal{F}_l}^{0,s} \hookrightarrow E_{2, \mathcal{F}_l}^{0,s} = H^0(R^s\Phi_{\mathcal{P}}D_A(\mathcal{F}_l)\otimes H_l).
$$
Then consider the spectral sequence
$$E_{2, \mathcal{F}_m}^{r,s}:= R^r\Gamma(R^{s}\Phi_{\mathcal{P}}D_A(\mathcal{F}_m)\otimes H_l) \Rightarrow R^{r+s}\Gamma(R\Phi_{\mathcal{P}}D_A(\mathcal{F}_m)\otimes H_l).$$
Since $E_{2, \mathcal{F}_m}^{r,s} = 0$ whenever $r < 0$, we get a natural homomorphism
$$\gamma_{\mathcal{F}_k}^s: R^{s}\Gamma(R\Phi_{\mathcal{P}}D_A(\mathcal{F}_m)\otimes H_l)) \to E_{\infty, \mathcal{F}_m}^{0,s} \hookrightarrow E_{2, \mathcal{F}_m}^{0,s} \cong H^0(R^{s}\Phi_{\mathcal{P}}D_A(\mathcal{F}_m)\otimes H_l).$$
For $i> 0$, by the condition (b) and the isomorphism $R^{-i}\Gamma(R\Phi_{\mathcal{P}}D_A(\mathcal{F}_m)\otimes H_l) \cong H^{i}(\mathcal{F}_m\otimes \hat{H_l}^*)^* = 0$ in Proposition \ref{fm-gv-sh} (1), we see that the following natural map is zero
$$\beta^{-i}: R^{-i}\Gamma(R\Phi_{\mathcal{P}}D_A(\mathcal{F}_l)\otimes H_l) \to R^{-i}\Gamma(R\Phi_{\mathcal{P}}D_A(\mathcal{F}_m)\otimes H_l) = 0.$$
Chasing in the following commutative diagram
$$\xymatrix{
&R^{s}\Gamma(R\Phi_{\mathcal{P}}D_A(\mathcal{F}_l)\otimes H_l) \ar[r]^{\gamma_{\mathcal{F}_l}^s}_{\cong}\ar[d]^{\beta^{s}} &H^0(R^s\Phi_{\mathcal{P}}D_A(\mathcal{F}_l)\otimes H_l)\ar[d]^{\bar{\beta}^{s}}  \\
&R^{s}\Gamma(R\Phi_{\mathcal{P}}D_A(\mathcal{F}_m)\otimes H_l)) \ar[r]^{\gamma_{\mathcal{F}_m}^s} &H^0(R^{s}\Phi_{\mathcal{P}}D_A(\mathcal{F}_m)\otimes H_l) \\
}$$
we can show that for $s \leq -1$, the map $\bar{\beta}^{s}$ is zero, consequently the map $R^{s}\Phi_{\mathcal{P}}D_A(\mathcal{F}_l) \to R^{s}\Phi_{\mathcal{P}}D_A(\mathcal{F}_m)$ is zero by the condition (a).
\medskip

To ease the notation, we denote $\mathcal{E}_m = \sigma_{\leq -1}R\Phi_{\mathcal{P}}D_A(\mathcal{F}_m)$ for $m=0,1,\cdots, d$. Then $\mathcal{E}_m \in D^{[-d,-1]}(\hat{A})$ by Proposition \ref{fm-gv-sh} (2). And since for every $i$,  $\mathcal{H}^i(\mathcal{E}_m) \to \mathcal{H}^i(\mathcal{E}_{m+1})$ is a zero map, we can show that for every $p$, the homomorphism $R^p\Psi_{\mathcal{P}}(\mathcal{H}^i(\mathcal{E}_m)) \to R^p\Psi_{\mathcal{P}}(\mathcal{H}^i(\mathcal{E}_{m+1}))$ is zero.
\medskip

Step 2: We prove that \emph{for $0< m \leq d$, the natural homomorphism
$$\alpha_k: R^0\Psi_{\mathcal{P}}(\sigma_{> -m -1}\mathcal{E}_0) \to R^0\Psi_{\mathcal{P}}(\sigma_{> -m -1}\mathcal{E}_m)$$
is zero}, which completes the proof of Lemma \ref{zeromap} if setting $m=d$.

We prove this assertion by induction. First note that for every $m$ and $0\leq i \leq d$, the object $\sigma_{\leq -m}(\sigma_{> -m -1}\mathcal{E}_i) \in D^b(\hat{A})$ is quasi-isomorphic to $\mathcal{H}^{-m}(\mathcal{E}_i)[k]$. When $m =1$, the map $\alpha_1: R^0\Psi_{\mathcal{P}}(\sigma_{> -2}\mathcal{E}_0) \to R^0\Psi_{\mathcal{P}}(\sigma_{> -2}\mathcal{E}_1)$ coincides with the map $R^{1}\Psi_{\mathcal{P}}(\mathcal{H}^{-1}(\mathcal{E}_0)) \to R^{1}\Psi_{\mathcal{P}}(\mathcal{H}^{-1}(\mathcal{E}_1))$, hence it is zero. Now assume that $\alpha_m$ is a zero map for some $m<d$. To prove $\alpha_{m+1}$ is zero, we consider the following commutative diagram
$$\xymatrix@C=0.5cm@R=0.5cm{
&\mathcal{H}^{-m-1}(\mathcal{E}_0)[m+1]\ar[r]\ar[d]  &\sigma_{> -m -2}\mathcal{E}_0\ar[r]\ar[d]
& \sigma_{> -m -1}\mathcal{E}_0\ar[d]\ar[r] &\mathcal{H}^{-m-1}(\mathcal{E}_0)[m+2]\ar[d]\\
&\mathcal{H}^{-m-1}(\mathcal{E}_m)[m+1]\ar[r]\ar[d]  &\sigma_{> -m -2}\mathcal{E}_m\ar[r]\ar[d]
& \sigma_{> -m -1}\mathcal{E}_m\ar[d]\ar[r] &\mathcal{H}^{-m-1}(\mathcal{E}_m)[m+2]\ar[d]\\
&\mathcal{H}^{-m-1}(\mathcal{E}_{m+1})[m+1]\ar[r]  &\sigma_{> -m -2}\mathcal{E}_{m+1}\ar[r]
& \sigma_{> -m -1}\mathcal{E}_{m+1}\ar[r]&\mathcal{H}^{-m-1}(\mathcal{E}_{m+1})[m+2]\\
}
$$
where the horizontal sequences are triangles.
Applying the right derived functor $R\Psi_{\mathcal{P}}$ to the above diagram induces
$$\xymatrix@C=0.6cm@R=0.6cm{
&R^{m+1}\Psi_{\mathcal{P}}(\mathcal{H}^{-m-1}(\mathcal{E}_0))\ar[r]\ar[d]^{\gamma_m}  &R^{0}\Psi_{\mathcal{P}}(\sigma_{> -m -2}\mathcal{E}_0)\ar[r]\ar[d]^{\beta_m}
& R^{0}\Psi_{\mathcal{P}}(\sigma_{> -m -1}\mathcal{E}_0)\ar[d]^{\alpha_m}\\
&R^{m+1}\Psi_{\mathcal{P}}(\mathcal{H}^{-m-1}(\mathcal{E}_m))\ar[r]^>>>>{\mu_m}\ar[d]^{\delta_m}  &R^{0}\Psi_{\mathcal{P}}(\sigma_{> -m -2}\mathcal{E}_m)\ar[r]\ar[d]^{\beta'_m}
& R^{0}\Psi_{\mathcal{P}}(\sigma_{> -m -1}\mathcal{E}_m)\ar[d]\\
&R^{mk+1}\Psi_{\mathcal{P}}(\mathcal{H}^{-m-1}(\mathcal{E}_{m+1}))\ar[r]  &R^{0}\Psi_{\mathcal{P}}(\sigma_{> -m -2}\mathcal{E}_{m+1})\ar[r]
& R^{0}\Psi_{\mathcal{P}}(\sigma_{> -m -1}\mathcal{E}_{m+1})\\
}
$$
where the horizontal sequences are exact. Since $\alpha_m$ is assumed to be a zero map, we have $\mathrm{Im}(\beta_m) \subseteq \mathrm{Im}(\mu_m)$. And since $\delta_m$ is zero, from the above commutative diagram we can conclude $\mathrm{Im}(\alpha_{m+1} = \beta'_m \circ \beta_m) \subseteq \mathrm{Im}(\beta'_m \circ \mu_m) =0$. Therefore, $\alpha_{m+1}$ is a zero map, and the proof is completed.
\end{proof}

\section{Subadditivity of Kodaira dimensions}\label{sec-pf-Iitaka}
In this section, we work over an algebraically closed field $k$ with $\mathrm{char}~k = p >5$. We will prove the following result on subadditivity of Kodaira dimensions.

\begin{thm}\label{Iit-conj}
Let $f: X \to Y$ be a fibration from a $\Q$-factorial projective 3-fold to a smooth projective variety of dimension $1$ or $2$. Let $B$ be an effective $\Q$-divisor on $X$ such that $(X,B)$ is klt.
Assume that $Y$ is of maximal Albanese dimension, and  and assume moreover that
\begin{itemize}
\item[$\spadesuit$]
if $\kappa(X_{\eta}, K_{X_{\eta}} + B_{\eta}) = \dim X/Y - 1$, then $B$ does not intersect the generic fiber $X_{\xi}$ of the relative Iitaka fibration
$I: X \dashrightarrow Z$ induced by $K_X +B$ on $X$ over $Y$.
\end{itemize}
Then
$$\kappa(X, K_X + B) \geq \kappa(X_{\eta}, K_{X_{\eta}} + B_{\eta}) + \kappa(Y).$$
\end{thm}

To prove the theorem above, we will first treat three subcases in the following theorems, which can be seen as complements of the results of \cite{Zh16a}.

\begin{thm}\label{fib-to-ab-var}
Let $(X,B)$ be a projective $\Q$-factorial klt pair of dimension 3. Let $f: X \to Y=A$ be a fibration to an elliptic curve or a simple abelian surface. Assume that $K_X + B$ is $f$-big. Then
$$\kappa(K_X + B) \geq \kappa(X_{\eta}, K_{X_{\eta}} + B_{\eta}).$$
\end{thm}

\begin{thm}\label{fib-to-curve-k=1}
Let $(X,B)$ be a projective $\Q$-factorial klt pair of dimension 3. Let $f: X \to Y$ be a fibration to a normal curve $Y$ of genus $g(Y) \geq 1$. Assume $\kappa(X_{\eta}, (K_X + B)|_{X_{\eta}}) = 1$ and the condition $\spadesuit$ holds. Then
$$\kappa(X, K_X + B) \geq 1 + \kappa(Y).$$
\end{thm}

\begin{thm}\label{fib-to-curve}
Let $(X,B)$ be a projective $\Q$-factorial klt pair of dimension 3. Let $f: X \to Y$ be a fibration to a normal curve $Y$ of genus $g(Y) \geq 1$. Assume $\kappa(X_{\eta}, (K_X + B)|_{X_{\eta}}) = 0$. Then
$$\kappa(X, K_X + B) \geq \kappa(Y).$$
If moreover $K_X +B$ is nef then it is semi-ample.
\end{thm}

We remark that Theorem \ref{fib-to-curve-k=1} is a generalization of \cite[Theorem 1.2, the subcase $\dim Y =1$ and $\kappa(X_{\eta})=1$]{EZ16} where the cases without boundary were treated. Since the condition $\spadesuit$ is assumed, we can adapt the arguments of \cite[Sec. 4]{EZ16} to our situation. For the convenience of the reader, we will provide a detailed proof.

\subsection{Preparations}
First let us recall an invariant introduced by Ejiri \cite[Sec.4]{Ej15} to measure the positivity of a sheaf.
\begin{defn}
Let $Y$ be a projective variety, $\mathcal{F}$ a torsion free coherent sheaf and $H$ an ample $\mathbb{Q}$-Cartier $\mathbb{Q}$-divisor on $Y$. Let
\begin{align*}
t(Y,\mathcal{F}, H)= \sup\{a \in \mathbb{Q}|\mathrm{the~sheaf~}&(F_{Y}^{e*}\mathcal{F})\otimes\mathcal{O}_{Y}(\llcorner -p^eaH \lrcorner) \\
~&\mathrm{is~ generically~globally~generated~ for ~some ~} e>0\}.
\end{align*}
\end{defn}

We will denote $t(Y, \mathcal{F}) \geq 0$ if $t(Y, \mathcal{F}, H) \geq 0$. This property is independent of the choices of $H$ (\cite[Remark 2.13]{Zh16a}) and is stronger than weak positivity (\cite[Prop. 4.7]{Ej15}).

From the main results of \cite{Zh16a} we deduce the following theorem.

\begin{thm}\label{thm-Zh16a}
Let $f:X\rightarrow Y$ be a separable fibration between smooth projective varieties, and let $D$ be a nef and $f$-big Cartier divisor on $X$.

(1) If $D$ is $f$-semi-ample, then

(1.a) for sufficiently divisible positive integers $n$ and $g$, the sheaf
$F_Y^{g*}f_*\mathcal{O}_X(nD + K_{X/Y})$ contains a nonzero subsheaf $V_n$ with $t(Y, V_n) \geq 0$, and $\mathrm{rk}~V_n\geq c n^{\dim X/Y}$ for some $c>0$ independent of $n$; and

(1.b) for any sufficiently divisible $n>0$ and big $\Q$-divisor $H$ on $Y$, $nD + K_{X/Y} + f^*H$ is big.

(2) If $K_Y$ is big, then for any sufficiently divisible $n>0$, $nD + K_{X}$ is big.

\end{thm}
\begin{proof}
We can assume $D = A+ E$ where $A$ is $f$-ample and $E$ is effective. For an integer $n>0$ such that $nE_{\ol\eta}$ is Cartier, let $s_{n}$ be a global section of $\O_{X_{\bar{\eta}}}(nE_{\ol\eta})$ with $(s_n)_0 = nE_{\ol\eta}$. We get an inclusion $S^{0}(X_{\ol\eta}, (nA + K_X)_{\ol\eta})\hookrightarrow S^{0}(X_{\ol\eta}, (nD + K_X)_{\ol\eta})$ by tensoring with $s_{n}$. Then by applying Lemma \ref{stable-sec-lem} on $X_{\ol\eta}$ for the divisor $(nA+K_X)_{\ol\eta}$, we can show that there exists some $c>0$ such that for any sufficiently divisible $n$,
$$\dim_{k(\bar{\eta})} S^{0}(X_{\ol\eta}, (nD + K_X)_{\ol\eta}) \geq \dim_{k(\bar{\eta})} S^{0}(X_{\ol\eta}, (nA + K_X)_{\ol\eta}) \geq c n^{\dim X/Y}.$$


The assertion (1.a) follows from applying \cite[Theorem 1.11]{Zh16a}. For (1.b), fix a sufficiently divisible integer $n>0$ such that $F_Y^{g*}f_*\mathcal{O}_X(nD + K_{X/Y})$ contains a nonzero subsheaf $V_n$ with $t(Y, V_n) \geq 0$. Applying \cite[Theorem 4.1]{Zh16a} shows that
$$\kappa(X, nD + K_{X/Y} + f^* H) \geq \kappa(X_{\eta}, (nD + K_{X/Y})|_{X_{\eta}}) + \dim Y = \dim X,$$
hence $nD + K_{X/Y} + f^* H$ is big.

For (2), take an ample divisor $H'$ on $Y$ such that $K_Y - H'$ is big. We can assume $D+ f^*H' \sim_{\Q}A' + \Delta$ where $A'$ is an ample divisor and $\Delta$ is an effective divisor with index not divisible by $p$. Since
$nD + K_X - K_{X/Y} - \Delta - f^*(K_Y - H') \sim_{\Q} (n-1)D + A'$ is nef and $f$-ample, and $\dim_{k(\bar{\eta})} S^0_{\Delta_{\bar{\eta}}}(X_{\bar{\eta}}, (nD + K_X)_{\bar{\eta}}) > 0$ for sufficiently divisible $n$, we can prove (2) by applying \cite[Theorem 1.5]{Zh16a}.
\end{proof}

\begin{cor}\label{cor-Zh16a}
Let $(X,B)$ be a projective $\Q$-factorial klt pair of dimension 3. Let $f: X \to Y$ be a fibration to a smooth projective curve or surface $Y$ of maximal Albanese dimension.
Assume that $K_X + B$ is $f$-big. Then
$$\kappa_{\sigma}(X, K_X + B) \geq \kappa(X_\eta, (K_X + B)_{\eta}) + \kappa(Y).$$
In particular, if $K_Y$ is big, then $K_X + B$ is big.


\end{cor}
\begin{proof}
Let $(\bar{X}, \bar{B})$ be a log minimal model of $(X,B)$ over $\mathrm{Alb}(Y)$.
And let $\rho: \tilde{X} \to X$ be a log resolution such that the natural map $\mu: \tilde{X} \to \bar{X}$ is a morphism. Let $D = \mu^*(K_{\bar{X}}+\bar{B})$. By Theorem \ref{rel-mmp} (3.1) and (3.4), $(\bar{X}, \bar{B})$ is in fact minimal, and $D$ is a nef divisor relatively big and semi-ample over $\mathrm{Alb}(Y)$ (hence over $Y$).
By Proposition \ref{num-prop}, we only need to prove $\nu(\tilde{X}, D) \geq \dim X_\eta + \kappa(Y)$.

First we reduce to separable fibrations by the following commutative diagram
\[\xymatrix@C=2cm{&X''\ar[r]^{\nu}\ar[drr]_{f''} &X'\ar[dr]^{f'} &\tilde{X}\ar[l]_{\sigma}\ar[d]^{\tilde{f}}\ar[r]^{\mu}    &\bar{X}\ar[d]\\
& &     &Y\ar[r]^{\mathrm{alb}_Y}     &\mathrm{Alb}(Y)\\
} \]
where $\sigma: \tilde{X} \to X'$ is a purely inseparable morphism constructed in Proposition \ref{red-to-sep} such that $f': X' \to Y$ is separable, $\nu: X'' \to X'$ is a smooth resolution of singularities, and $\tilde{f}, f', f''$ denote natural induced morphisms.

Since $\sigma$ is purely inseparable, there exists $D'$ on $X'$ such that $\sigma^*D' =D$. Set $D'' = \nu^*D'$, which is a nef divisor relatively big and semi-ample over $Y$. Take a big divisor $H$ on $Y$. By Theorem \ref{thm-Zh16a}, for sufficiently divisible $n>0$, the divisor $nD'' + K_{X''} - f''^*K_Y + f''^*H$ is big. By Proposition \ref{red-to-sep}, there exist a rational number $t>0$ and an effective divisor $\Delta'$ such that $K_{\tilde{X}} \sim_{\Q} \sigma^*t(K_{X'} + \Delta')$. And we can write that $K_{X''} = \nu^*(K_{X'} + \Delta') + E'' - F''$ where $E'', F''$ are effective divisors on $X''$ and $E''$ is $\nu$-exceptional.
Applying Theorem \ref{ct}, we conclude that
\begin{equation}\label{choice-H}
\small\begin{split}
&\kappa(\tilde{X}, tnD + K_{\tilde{X}} -t\tilde{f}^*K_Y + t\tilde{f}^*H)\\
&\geq \kappa(X', t(nD' -f'^*K_Y + f'^*H) + t(K_{X'} + \Delta')) \\
&= \kappa(X'', \nu^*t(nD' -f'^*K_Y + f'^*H) + \nu^*t(K_{X'} + \Delta')  + tE'') \\
&\geq \kappa(X'', t(nD'' + K_{X''} - f''^*K_Y + f''^*H)) \geq 3.
\end{split}
\end{equation}
Since this is true for any big $\Q$-divisor $H$, and $D$ is nef, we conclude that for $q\gg0$, $qD + K_{\tilde{X}} -t\tilde{f}^*K_Y$ is pseudo-effective.

When $\kappa(Y)=0$, if $\dim X_\eta =1$ then $\nu(\tilde{X}, D) \geq \dim X_\eta$ since $D$ is nonzero and nef; if $\dim X_\eta =2$, then for a general fiber $\tilde{F}$ of $\tilde{f}$, $D^2 \cdot \tilde{F}>0$ since $D$ is nef and $\tilde{f}$-big, thus $\nu(\tilde{X}, D) \geq 2$.

When $K_Y$ is big, by setting $H = K_Y$ in Eq. (\ref{choice-H}), we obtain that $nD + K_{\tilde{X}}$ is big. As we can write that $K_{\tilde{X}} = \mu^*(K_{\bar{X}}+\bar{B}) + E -F$ where $E, F$ are effective divisors on $\tilde{X}$ and $E$ is $\mu$-exceptional, applying Theorem \ref{ct}, it follows that $(n+1)D$ is big.

It remains to consider the case $\kappa(Y) =1$ and $\dim Y=2$. Take an ample divisor $\bar{A}$ on $\bar{X}$. We only need to prove that $D^2 \cdot \mu^*\bar{A} >0$. Fix a $q\gg 0$ such that $qD + K_{\tilde{X}} -t\tilde{f}^*K_Y$ is pseudo-effective. Then $D^2 \cdot (qD + K_{\tilde{X}} -t\tilde{f}^*K_Y) \geq 0$.
And since $E$ is $\mu$-exceptional, by projection formula we have
$\mu^*(K_{\bar{X}}+\bar{B}) \cdot \mu^*\bar{A} \cdot E = 0$.
Take a general divisor $H' \in |NK_Y|$ for some sufficiently divisible $N$ and set $\tilde{H}' = \tilde{f}^*H'$. Then $\tilde{H}'$ contains a component $\tilde{H}$, such that $\mu^*\bar{A}|_{\tilde{H}}$ is semi-ample and big and that $D|_{\tilde{H}}$ is nef and $\tilde{f}|_{\tilde{H}}$-big. Therefore,
$$D\cdot \mu^*\bar{A} \cdot \tilde{f}^*K_Y = \frac{1}{N}D\cdot \mu^*\bar{A} \cdot \tilde{H}' \geq \frac{1}{N}D\cdot \mu^*\bar{A} \cdot \tilde{H} = \frac{1}{N}(D|_{\tilde{H}})\cdot (\mu^*\bar{A}|_{\tilde{H}})>0$$
where the last strict inequality is obtained by applying Hodge Index Theorem.
Finally the proof is completed by
\begin{align*}
(q+1)D^2 \cdot \mu^*\bar{A} &= D\cdot (qD + \mu^*(K_{\bar{X}}+\bar{B})) \cdot \mu^*\bar{A} =  D\cdot (qD + K_{\tilde{X}}-E  +F) \cdot \mu^*\bar{A} \\
&\geq D\cdot (qD + K_{\tilde{X}} - t\tilde{f}^*K_Y + t\tilde{f}^*K_Y) \cdot \mu^*\bar{A} >0.
\end{align*}
\end{proof}

Recall a positivity result on surfaces.
\begin{lem}\textup{(\cite[Lemma 2.11]{EZ16})}\label{positivity-surf}
Let $g:Z\to Y$ be a generically smooth fibration from a smooth projective surface to a smooth projective curve. Let $H$ be a nef and $g$-big divisor on $Z$. Then $g_*\O_Z(K_{Z/Y}+lH)$ is a nef vector bundle for every $l\gg 0$.
\end{lem}

The following result was proved by Waldron \cite{Wal15} when $\kappa =2$, and the case $\kappa =1$ follows easily from applying Theorem \ref{rel-mmp} (3.2).
\begin{thm}\textup{(\cite[Theorem 3.1]{Zh16c})}\label{s-amp-k=2}
Let $(X, B)$ be a $\Q$-factorial klt projective 3-fold. Assume that $K_X + B$ is nef. If $\kappa(X, K_X + B) \geq 1$, then $K_X + B$ is semi-ample.
\end{thm}

We extract the following lemma from the strategy of \cite[Sec. 4]{EZ16}, which will be used in the proof of Theorem \ref{fib-to-ab-var} and \ref{fib-to-curve-k=1}.
\begin{lem}\label{tor-lem}
Let $(\hat{X},\hat{B})$ be a minimal projective $\Q$-factorial dlt pair of dimension 3, and let $\hat{f}: \hat{X} \to Y$ be a fibration to a normal variety. Assume that

$(a)$ $\hat{B} = G_1 + G_2 + \cdots + G_n$ is a sum of prime Weil divisors.\\
Then for every $j=1,2,\cdots, n$, $(K_{\hat{X}} + \hat{B})|_{G_j}$ is semi-ample, and a general fiber $F_j$ of the Iitaka fibration induced by $(K_{\hat{X}} + \hat{B})|_{G_j}$ is integral.

Assume moreover that

$(b)$ there exist $N >0$ and two different effective Cartier divisors $\hat{D}_i, i=1,2$ such that
$\hat{D}_i \sim N(K_{\hat{X}} + \hat{B}) + \hat{f}^*L_i$ for some $L_i \in \Pic^0(Y)$ and that
$$\mathrm{Supp}\hat{D}_i \subseteq \mathrm{Supp}\hat{B};$$

$(c)$ there exist effective divisors $\hat{G}_1, \hat{G}_2, \hat{G}'_1, \hat{G}'_2$ such that $\hat{D}_1 = a_{11}\hat{G}_1 + a_{12}\hat{G}_2 + \hat{G}'_1~\mathrm{and}~ \hat{D}_2 = a_{21}\hat{G}_1 + a_{22}\hat{G}_2 + \hat{G}'_2$ where $a_{11} >a_{21} \geq 0$ and $a_{22} > a_{12} \geq 0$; and

$(d)$  there exist two irreducible components, say, $G_1,G_2$ of $\hat{G}_1, \hat{G}_2$ respectively, such that for $i, j \in \{1,2\}$ and $i\neq j$, $F_j$ is dominant over $Y$ and
$$F_j \cap \mathrm{Supp}~(\hat{G}_j'':=\hat{G}_i+ \hat{G}'_1+ \hat{G}'_2)=\emptyset.$$
Then both $L_1$ and $L_2$ are torsion.

Furthermore, condition (d) holds, if for $j=1,2$, $G_j$ is not a component of $\hat{G}_j''$ and $\kappa(F_j) \geq 0$.
\end{lem}
\begin{proof}

Note that since each $G_i$ is a dlt center of $(\hat{X}, \hat{B})$, $G_i$ is a normal surface by \cite[Lemma 4.2]{Bir16}. Moreover we have $(\hat{B} - G_i)|_{G_i} \geq 0$, and $(K_{\hat{X}} + \hat{B})|_{G_i} = K_{G_i} + (\hat{B} - G_i)|_{G_i}$ is log canonical. Then since $K_{\hat{X}} + \hat{B}$ is assumed nef, by \cite[Theorem 1.2]{Ta14}, $(K_{\hat{X}} + \hat{B})|_{G_i}$ is semi-ample. And as $\dim G_i = 2$, $F_j$ is always integral by Proposition \ref{sep-fib}.

Assume (b), (c) and (d). Let's prove that $L_1, L_2$ are torsion. First considering the restrictions on $F_1$, by $(K_{\hat{X}} + \hat{B})|_{F_1} \sim_{\Q} 0$, applying these assumptions we can show
\begin{align*}
a_{21}\hat{f}^*L_1|_{F_1} &\sim_{\mathbb{Q}} a_{21}(N(K_{\hat{X}} + \hat{B}) + \hat{f}^*L_1)|_{F_1}\\
&\sim a_{21}\hat{D}_1|_{F_1} \sim a_{11}a_{21}\hat{G}_1|_{F_1} \sim a_{11}\hat{D}_2|_{F_1} \\ &
\sim a_{11}(N(K_{\hat{X}} + \hat{B}) + \hat{f}^*L_2)|_{F_1} \sim_{\mathbb{Q}} a_{11}\hat{f}^*L_2|_{F_1},
\end{align*}
thus $a_{21}L_1 \sim_{\mathbb{Q}} a_{11}L_2$ by Lemma \ref{inj-pic}. And restricting on $F_2$, in the same way we can show
$a_{22}L_1 \sim_{\mathbb{Q}} a_{12}L_2$. Then granted these two relations, we can conclude $L_i \sim_{\Q} 0, i=1,2$ from the fact that $\begin{pmatrix}a_{11} &a_{12}\\a_{21} &a_{22} \end{pmatrix}$ is invertible over $\Q$, which is because $a_{11} >a_{21} \geq 0$ and $a_{22} > a_{12} \geq 0$.

It remains to prove the third assertion. As $\kappa(F_j) \geq 0$, we can assume the canonical divisor $K_{F_j'} \geq 0$ where $F_j'$ is the normalization of $F_j$. Applying the adjunction formula, we get
\begin{align*}
0\sim_{\Q} (K_{\hat{X}} + \hat{B})|_{F_j'} &\sim_{\Q} ((K_{\hat{X}} + \hat{B})|_{G_j})|_{F_j'} \\
&\sim_{\Q} (K_{G_j} + (\hat{B} - G_j))|_{F_j'} \sim_{\Q} K_{F'_j} + C_j' + (\hat{B} - G_j)|_{F_j'}
\end{align*}
where $C_j' \geq 0$ on $F_j'$. As $F_j$ is general, we may assume $F_j$ is not contained in $\hat{B} - G_j$. In turn we conclude that $(\hat{B} - G_j)|_{F_j'} =0$. This, combing with the assumption that $G_j$ is not a component of $\hat{G}_j''$, indicates that $F_j \cap \mathrm{Supp}~\hat{G}_j''=\emptyset$.
\end{proof}

\subsection{Proof of Theorem \ref{fib-to-ab-var}}

For a sufficiently large integer $e$, the Weil index of $B' = \frac{p^e}{p^e+1}B$ is not divisible by $p$, and  $K_X + B'$ is still $f$-big. Replacing $B$ with $B'$, we can assume the Weil index of $B$ is not divisible by $p$.
By Theorem \ref{rel-mmp} (3.1, 3.4), we can replace $X$ with the relative log canonical model over $A$ with the loss of $X$ being $\Q$-factorial, thus $K_X + B$ is a nef and $f$-ample $\Q$-Cartier $\Q$-divisor.
\smallskip

We claim that \emph{we only need to prove $\kappa(K_X +B) \geq 1$}. Indeed, if this is true then $K_X +B$ is semi-ample by Theorem \ref{s-amp-k=2}, thus for a sufficiently divisible $M>0$, the linear system $|M(K_X +B)|$ has no base point. Since $(K_X +B)_{\eta}$ is big, the restriction $|M(K_X +B)||_{X_{\eta}}$ on the generic fiber defines a generically finite morphism, which indicates that $\kappa(X, K_X +B) \geq \kappa(X_{\eta}, (K_X +B)_{\eta})$.
\smallskip

Let $l, g> 0$ be two integers such that $l(K_X + B)$ is Cartier and $(p^g - 1)B$ is integral.
For an integer $e>0$ divisible by $g$, we have the trace map
$$Tr_{X,B}^{e,l}: \mathcal{F}_{e,l}:=f_*(F_{X*}^{e}\O_X((1-p^e)(K_X + B)) \otimes \O_X(l(K_X + B)))  \to f_*\O_X(l(K_X + B)),$$
and denote its image by $\mathcal{F}_e^l$. The restriction of $\mathcal{F}_e^l$ on the generic point $\eta$ defines a linear system $|(\mathcal{F}_e^l)_{\eta}|$ contained in $|l(K_X + B)_{\eta}|$.

\smallskip

We claim the following assertions with the proof postponed.
\begin{itemize}
\item[C1]
\emph{For every $l >0$ such that $l(K_X + B)$ is Cartier, there exists some $e(l) >0$ such that for any $e \geq e(l)$ divisible by $g$, the rank of $\mathcal{F}_e^l$ is a stable number $r_l$; and there exists an integer $K >0$ such that, for any $l$ divisible by $K$ and any $e$ divisible by $g$, the linear system $|(\mathcal{F}_e^l)_{\eta}|$ defines a generically finite map.}
\item[C2]
\emph{If $L$ is an ample line bundle on $\hat{A}$ and $l\geq 2$ is an integer such that $l(K_X + B)$ is Cartier, then for any $i>0$ and sufficiently divisible integer $e>0$,
$$H^i(A, \mathcal{F}_{e,l} \otimes \hat{L}^*) = 0.$$}
\end{itemize}

\smallskip

Fix a positive integer $l$ divisible by $K$ and an integer $e_0$ such that $\mathrm{rk}~\mathcal{F}_{e_0}^l = r_l$. Let $\mathcal{F} = \mathcal{F}_{e_0}^l$. Then by (C1), the linear system $|\mathcal{F}_{\eta}|$ defines a generically finite map of $X_{\eta}$. Recall that for positive integers $e'>e$ divisible by $g$, there exists a natural trace map $\mathcal{F}_{e',l} \to \mathcal{F}_{e,l}$ (Sec. \ref{tmaF}). Applying (C2), by induction we can find two integers $e_1 < e_2$ divisible by $g$ and bigger than $e_0$, two ample line bundles $H_0, H_1$ on $\hat{A}$ and three sheaves $\mathcal{F}_{0} := \mathcal{F}, \mathcal{F}_{1} := \mathcal{F}_{e_1,l}, \mathcal{F}_{2} := \mathcal{F}_{e_2,l}$ which satisfy the conditions of Theorem \ref{CGG}. Note that if $\dim A=1$ we only need two sheaves $\mathcal{F}_{0}, \mathcal{F}_{1}$.
Applying Theorem \ref{CGG} (i) and Proposition \ref{fm-gv-sh} (2), the cohomological locus $V^0(\mathcal{F}) = (-1)_{\hat{A}}^*\mathrm{Supp} R^0\Phi_{\mathcal{P}}D_A(\mathcal{F})\neq \emptyset$.

\smallskip

The statement is trivial when $\nu(K_X +B) =3$, so from now on
we assume $\nu(K_X +B) \leq 2$. We break the proof into three steps.

\smallskip

Step 1: We prove that if $\dim V^0(\mathcal{F}) >0$, then $\kappa(K_X +B) \geq 1$.

In this case $V^0(\mathcal{F})$ generates $\hat{A}$ since $\hat{A}$ is simple. Hence so does $V^0(f_*\O_X(l(K_X + B)))$ too, as $V^0(\mathcal{F}) \subseteq V^0(f_*\O_X(l(K_X + B)))$. So for $m \geq \dim A$, the map
$$\times^m V^0(f_*\O_X(l(K_X + B))) \to \hat{A}~\mathrm{via}~(\alpha_1, \alpha_2, \cdots, \alpha_m) \mapsto \alpha_1 + \alpha_2 + \cdots + \alpha_m$$
is surjective. In particular if $m> 2\dim A$, there exist infinitely many $m$-tuples $(\alpha_1, \alpha_2, \cdots, \alpha_m)$ mapped to $\hat{0}$.
Consider the natural map
\begin{align*}
\textstyle{H^0(X, l(K_X + B)}& \textstyle{+ f^*\alpha_1) \times H^0(X, l(K_X + B) + f^*\alpha_2) \times \cdots \times H^0(X, l(K_X + B) +  f^*\alpha_m)} \\
& \textstyle{\to H^0(X, ml(K_X + B)+ f^*(\alpha_1 + \alpha_2 \cdots + \alpha_m)) \cong H^0(X, ml(K_X + B)).}
\end{align*}
We can show $h^0(X, ml(K_X + B)) \geq 2$, thus $\kappa(K_X +B) \geq 1$.

\smallskip

From now on we assume that $V^0(\mathcal{F}) = (-1)_{\hat{A}}^*\mathrm{Supp} R^0\Phi_{\mathcal{P}}D_A(\mathcal{F})$ consists of finitely many closed points.

\smallskip

Step 2: We will find an integer $m_1$ and some divisors $D_i \in |m_1(K_X + B) + f^*L_i|, i=1,2,\cdots, r$ for some $L_i \in \mathrm{Pic}^0(A)$, such that the sub-linear system of $|m_1(K_X + B)_{\eta}|$ generated by $(D_i)_{\eta}, i=1,2,\cdots, r$ defines a generically finite map of $X_{\eta}$.

By Theorem \ref{CGG} (ii), there exist an isogeny $\pi: A_1 \to A$, $P_1, P_2, \cdots, P_s \in \mathrm{Pic}^0(A_1)$ and a generically surjective homomorphism $\bigoplus_j P_j \to \pi^*\mathcal{F}$. Consider the following commutative diagram
$$\centerline{\xymatrix@C=1cm{ &X_1'\ar[rd]_{f_1'}\ar[r]_<<<<<<{\nu} \ar@/^1pc/[rr]^{\pi_1'} &X_1 = X\times_A A_1 \ar[d]_{f_1} \ar[r]_<<<<<<{\pi_1} &X\ar[d]^{f}\\
                         &     &A_1\ar[r]^{\pi}                               &A
}}$$
where $X_1'$ is the normalization of the reduced
scheme structure of $X_1$. There exists a natural composition homomorphism by \cite[Chapter III. Prop. 9.3]{Har77}
\begin{equation*}
\small\begin{split}
\alpha: \bigoplus_j P_j \to \pi^*\mathcal{F} \hookrightarrow &\pi^*f_*\O_X(l(K_X + B)) \cong  f_{1*}\pi_1^*\O_X(l(K_X + B))\\
& \to f'_{1*}(\nu^*\pi_1^*\O_X(l(K_X + B))) \cong f'_{1*}(\pi_1'^*\O_X(l(K_X + B))).
\end{split}
\end{equation*}
The linear system $|(\mathrm{Im} \alpha)_{\eta}| \subseteq |\pi_1'^*(l(K_X + B))|_{(X'_1)_\eta}|$ defines a generically finite map of $(X'_1)_\eta$.
If the direct summand $P_j$ is not mapped to zero via $\alpha$, then $h^0(X_1', \pi_1'^*(l(K_X + B)) - f_1'^*P_j) \geq 1$.
Since $\pi^*: \mathrm{Pic}^0(A) \to \mathrm{Pic}^0(A_1)$ is an isogeny, there exist $Q_j \in \mathrm{Pic}^0(A)$ such that $P_j = \pi^*Q_j$. Applying
Theorem \ref{ct}, if $\alpha(P_j) \neq 0$ then
$$\kappa(X, l(K_X + B) - f^*Q_j) = \kappa(X', \pi_1'^*(l(K_X + B) - f_1'^*P_j) \geq 0.$$
We can find a sufficiently divisible integer $l_1>0$ such that for every $j$, if $\alpha(P_j) \neq 0$, the pull-back linear system $(\pi_1'^*|l_1(l(K_X + B) - f_1'^*P_j)|)|_{(X'_1)_\eta}$ defines a rational map $\phi_j: (X'_1)_{\eta} \dashrightarrow \mathbb{P}^{r_j}$ where $r_j = \dim |l_1((l(K_X + B) - f_1'^*P_j)|$, whose image has dimension $\kappa(X, l(K_X + B) - f^*Q_j)$. By the construction, the sub-linear system of $|\pi_1'^*(l_1l(K_X + B))_\eta|$ generated by the divisors of $(\pi_1'^*|l_1(l(K_X + B) - f_1'^*P_j)|)|_{(X'_1)_\eta}, j=1,2, \cdots, s$ defines a generically finite map of $(X'_1)_\eta$. Therefore, there exist $L_i = -l_1Q_{j_i}\in \mathrm{Pic}^0(A)$ for some $j_1,j_2,\cdots, j_r$ and effective divisors $D_i \in |m_1(K_X + B) + f^*L_i|$ where $m_1 = l_1l$, such that the linear system generated by $(D_i)_{\eta}, i=1,2,\cdots, r$ defines a generically finite map of $X_{\eta}$.

\smallskip

We aim to prove that there exist at least two different divisors among $D_i$, say, $D_1 \neq D_2$, such that $L_1, L_2$ are torsion in $\mathrm{Pic}^0(A)$. Then for some sufficiently divisible $N>0$ such that $NL_1 \sim NL_2 \sim 0$, we have $ND_1, ND_2 \in |Nm_1(K_X +B)|$, which concludes $\kappa(X, K_X +B) \geq 1$.

\smallskip

Step 3: We will construct a minimal dlt pair $(\hat{X}, \hat{B})$ and divisors $\hat{D}_1, \hat{D}_2$ satisfying the conditions of Lemma \ref{tor-lem}.

(3.1) Take a log resolution $\mu: \tilde{X} \to X$ of $B+ \sum_iD_i$. Let $\tilde{B}$ be the reduced divisor supported on the union of $\mu^{-1}(B+ \sum_iD_i)$ and the $\mu$-exceptional divisors. Then $(\tilde{X}, \tilde{B})$ is dlt, and $K_{\tilde{X}} + \tilde{B}$ has a weak Zariski decomposition.
By Theorem \ref{rel-mmp}, running a relative log MMP for $(\tilde{X}, \tilde{B})$ over $A$, we can get a dlt log minimal model $(\hat{X}, \hat{B})$ and a fibration
$\hat{f}: \hat{X} \to A$.
The divisor $\tilde{E} = K_{\tilde{X}} + \tilde{B} - \mu^*(K_X + B)$ is effective. Take a sufficiently divisible integer $l_2>0$ such that $l_2\tilde{E}$ is Cartier. Let $m_2 = m_1l_2$. We get effective divisors
$$\tilde{D}_i = l_2\mu^*D_i + l_2\tilde{E} \sim m_2(K_{\tilde{X}} + \tilde{B}) +  l_2\mu^*f^*L_i$$
and the push-forward divisors via the natural map $\tilde{X} \dashrightarrow \hat{X}$
$$\hat{D}_i \sim m_2(K_{\hat{X}} + \hat{B}) +  l_2\hat{f}^*L_i.$$

(3.2) We prove $\nu(K_{\hat{X}} + \hat{B}) = \nu(K_X +B)$ as follows.
Applying Proposition \ref{num-prop}, on one hand since $K_{\tilde{X}} + \tilde{B} \geq \mu^*(K_X +B)$ we have $\nu(K_{\hat{X}} + \hat{B}) = \kappa_{\sigma}(K_{\tilde{X}} + \tilde{B}) \geq \nu(K_X +B)$, on the other hand since there exists an effective $\mu$-exceptional divisor $\tilde{E}'$ such that $K_{\tilde{X}} + \tilde{B} \leq  \mu^*(K_X +B) + \sum_i\mu^*D_i + \tilde{E}' \equiv (rm_1+1)\mu^*(K_X +B) + \tilde{E}'$,
we conclude $ \kappa_{\sigma}(K_{\tilde{X}} + \tilde{B}) \leq \nu(K_X +B)$. In summary, we get the equality $\nu(K_{\hat{X}} + \hat{B}) = \nu(K_X +B)$.


(3.3) The restrictions of $\hat{D}_i$ on $\hat{X}_{\eta}$ generate a linear system $|\hat{V}| \subseteq |m_2(K_{\hat{X}} + \hat{B})_{\eta}|$, which defines a generically finite map $\hat{X}_{\eta} \dashrightarrow \mathbb{P}_{k(\eta)}^{r-1}$ by the construction in Step 2. Let $\hat{C}_{\eta}$ be the fixed part of $|\hat{V}|$, and set $\hat{A}_{i, \eta}  = (\hat{D}_{i})_\eta - \hat{C}_{\eta}$. We may assume $\hat{A}_{1, \eta} \neq \hat{A}_{2, \eta}$. Since $\hat{A}_{1, \eta} \sim \hat{A}_{2, \eta}$, we can choose two irreducible components $G_{1, \eta}, G_{2, \eta}$ of $\hat{A}_{1, \eta} + \hat{A}_{2, \eta}$ such that,
\begin{itemize}
\item
if $b_{ij}, i,j =1,2$ are the coefficients of $G_{j, \eta}$ in $\hat{A}_{i, \eta}$ respectively, i.e.,
$$\hat{A}_{1, \eta} = b_{11}G_{1, \eta} + b_{12}G_{2, \eta} + G_{1, \eta}'~\mathrm{and}~ \hat{A}_{2, \eta} = b_{21}G_{1, \eta} + b_{22}G_{2, \eta} + G_{2, \eta}'$$
where neither of $G_{1, \eta}, G_{2, \eta}$ are contained in $G_{1, \eta}' + G_{2, \eta}'$, then $b_{11} >b_{21} \geq 0$ and $b_{22} > b_{12} \geq 0$.
\end{itemize}

(3.4) If $\dim \hat{X}_{\eta} = 2$, since $K_{\hat{X}} + \hat{B}$ is relatively big over $A$, we can choose $\hat{A}_{2, \eta}$ and $G_{i, \eta}, i=1,2$ such that $(K_{\hat{X}} + \hat{B})|_{G_{i, \eta}}$ is big as follows. First we can choose a component $G_{1, \eta}$ of $\hat{A}_{1, \eta}$ such that $(K_{\hat{X}} + \hat{B})|_{G_{1, \eta}}$ is big. Second since $\hat{A}_{i, \eta}, i=1,2,\cdots, r$ have no common component, we can choose $\hat{A}_{2, \eta}$ not containing $G_{1, \eta}$, i.e., $b_{21} = 0$. Finally since $\hat{A}_{1, \eta} \sim \hat{A}_{2, \eta}$ and the intersection number $(K_{\hat{X}} + \hat{B})\cdot (\hat{A}_{1, \eta} - b_{11}G_{1, \eta}) < (K_{\hat{X}} + \hat{B})\cdot \hat{A}_{2, \eta}$, $\hat{A}_{2, \eta}$ must have an irreducible component $G_{2, \eta}$ such that $(K_{\hat{X}} + \hat{B}) \cdot G_{2, \eta} >0$ and the coefficients $b_{22} > b_{12}$.

(3.5) Let $G_i, i=1,2$ be the reduced irreducible divisors on $\hat{X}$ such that $(G_i)_{\eta} = G_{i, \eta}$. By (3.3) we can write that
\begin{equation*}\label{eq1}
\hat{D}_1 = a_{11}G_1 + a_{12}G_2 + G'_1~\mathrm{and}~ \hat{D}_2 = a_{21}G_1 + a_{22}G_2 + G'_2
\end{equation*}
where $a_{11} >a_{21} \geq 0$ and $a_{22} > a_{12} \geq 0$, and neither of $G_1,G_2$ are contained in $G_1' + G_2'$.
\medskip

(3.6) By Lemma \ref{tor-lem}, we know that $G_i$ is normal, and $(K_{\hat{X}} + \hat{B})|_{G_i}$ is semi-ample.
Since $\hat{D}_i \equiv m_2(K_{\hat{X}} + \hat{B})$ and $G_i$ is a component of $\hat{D}_i$, by Proposition \ref{num-prop} we have
$$\kappa(G_i, (K_{\hat{X}} + \hat{B})|_{G_i}) = \nu(G_i, (K_{\hat{X}} + \hat{B})|_{G_i}) \leq \nu(\hat{X}, K_{\hat{X}} + \hat{B}) - 1.$$

For the case $\dim A =2$, we have $\nu(G_i, (K_{\hat{X}} + \hat{B})|_{G_i}) =0$ or $1$. The case $\nu(G_i, (K_{\hat{X}} + \hat{B})|_{G_i}) = 1$ does not happen, because otherwise, the divisor $(K_{\hat{X}} + \hat{B})|_{G_i} \sim_{\Q} K_{G_i} + (\hat{B} - G_i)|_{G_i}$ will induce a fibration fibred by curves of arithmetic genus $\leq 1$ on $G_i$, which is impossible since $G_i$ is dominant over $A$ and $A$ is simple. Thus the Iitaka fibration induced by $(K_{\hat{X}} + \hat{B})|_{G_i}$ is trivial. By (3.5), the surface $G_i$ is dominant over $A$, hence $\kappa(G_i, K_{G_i}) \geq 0$ (see for example \cite[Prop. 13.1]{Ba01}).

For the case $\dim A =1$, $\hat{X}_\eta$ is a surface over $k(\eta)$, and by (3.4)
$$\kappa(G_i, (K_{\hat{X}} + \hat{B})|_{G_i}) = \nu(G_i, (K_{\hat{X}} + \hat{B})|_{G_i}) = 1.$$
Take a general fiber $F_i$ of the Iitaka fibration induced by $(K_{\hat{X}} + \hat{B})|_{G_i}$ on $G_i$. Since $(K_{\hat{X}} + \hat{B})|_{G_{i,\eta}}$ is big, $F_i$ is a curve dominant over $A$, thus $\kappa(F_i, K_{F_i}) \geq 0$.

\smallskip
Let $\hat{G}_i = G_i$ and $\hat{G}'_i = G'_i$ for $i=1,2$. By (3.1), (3.5) and (3.6), the conditions of Lemma \ref{tor-lem} are satisfied, hence $l_2L_1, l_2L_2$ are torsion, which finishes the proof.

\medskip

It remains to prove Claim (C1) and (C2).
\smallskip

(C1)
For every $l>0$, if $e' >e$ then $\mathcal{F}_{e'}^l \subseteq \mathcal{F}_e^l$, hence there exists some $e(l)$ such that for every $e \geq e(l)$, we have $\mathrm{rk}~\mathcal{F}_e^l = \mathrm{rk}~\mathcal{F}_{e(l)}^l$, which is denoted by $r_l$.

To study the linear system $|(\mathcal{F}_e^l)_{\eta}|$, we may replace $B$ with $B + f^*H$ for some ample divisor $H$ on $A$, so $K_X+B$ is ample. And we may replace $B$ with a bigger divisor $B + tD$ for some effective divisor $D \sim_{\Q} K_X+B$ and a rational number $t >0$, to make the Cartier index of $K_X + B$ not divisible by $p$.\footnote{The divisor $D$ can be obtained as follows. Take an effective Weil divisor $D_1 \sim K_X + A$ for some ample divisor $A$ on $X$. For sufficiently large $v$ the divisor $K_X + B - \frac{1}{p^v -1}(D_1 + B)$ is an ample $\Q$-Cartier $\Q$-divisor and hence is $\Q$-linearly equivalent to an effective divisor $D_2$, which can be assumed to have Cartier index not divisible by $p$. Let $D= D_2 +  \frac{1}{p^v -1}(D_1 + B)$. Identifying $K_X$ with $D_1 - A$ similarly as in \cite[Lemma 3.15]{Pa14}, one can verify that $p$ does not divide the Cartier index of
$$K_X + B +  \frac{1}{p^{v} + 1}D = -A + \frac{1}{p^{v} + 1}D_2 + \frac{p^{2v}}{p^{2v} - 1}(D_1 + B)$$}
Moreover we remark that for some fixed $l$, to show that the linear system $|(\mathcal{F}_e^l)_{\eta}|$ defines a generically finite map of $X_\eta$, we only need to prove this assertion for any sufficiently divisible $e$.

Fix a sufficiently divisible integer $g'>0$ such that $(1-p^{g'})(K_X + B)$ is Cartier, and that
for the non-$F$-pure ideal $\sigma(B)$ (Sec. \ref{tmaF})
$$Tr_{X,B}^{g'} (F_{X*}^{g'}(\sigma(B)\cdot\O_X((1-p^{g'})(K_X + B)))) = \sigma(B).$$
Since $K_X+B$ is ample, applying Lemma \ref{stable-sec-lem}, we can find a sufficiently divisible integer $K$ such that
for every positive integer $l$  divisible by $K$ the trace map
\begin{align*}\Phi_{e,l}: H^0(X, F_{X*}^{eg'}(\sigma(B)\cdot\O_X((1-p^{eg'})&(K_X + B)))\otimes \O_X(l(K_X + B)))\\
 &\to H^0(X, \sigma(B)\cdot \O_X(l(K_X + B))).\end{align*}
is surjective. It follows that
$$|\sigma(B) \cdot \O_X(l(K_X + B))||_{X_{\eta}} = |\mathrm{Im}(\Phi_{e,l})||_{X_{\eta}}.$$
If necessary, we can enlarge $K$ to assume that the linear system $|\sigma(B) \cdot \O_X(K(K_X + B))|$ defines a generically finite map of $X$. Then by $|\mathrm{Im}(\Phi_{e,l})||_{X_{\eta}} \subseteq |(\mathcal{F}_e^l)_{\eta}|$,
we show that $|(\mathcal{F}_e^l)_{\eta}|$ defines a generically finite map of $X_{\eta}$.

\smallskip

(C2) Before the proof we remind that the Weil index $q_0$ of $K_X + B$ is not divisible by $p$, but the Cartier index $q_1$ is possibly divisible by $p$. From now on to the end of the proof of (C2), we assume the integer $e$ always satisfies $q_0|p^e-1$. Let $d = \frac{q_1}{q_0}$. Then $H = q_1(K_X + B)$ is a nef and $f$-ample divisor, and we have a set consisting of finitely many coherent sheaves
$$\{\mathcal{G}_r= \O_X(rq_0(K_X + B))| r=0,1,\cdots d-1 \}.$$

Define $\phi_L: \hat{A} \to A ~\mathrm{via}~\hat{t} \mapsto T_{\hat{t}}^*L \otimes L^{-1}\in \mathrm{Pic}^0(\hat{A}) = A.$
Then by results of \cite[Sec. 13, 16]{Mum74}, the morphism $\phi_L$ is an isogeny since $L$ is ample, and
$$\phi^*\hat{L}^* \cong \bigoplus^{r} L~\mathrm{where}~r = h^0(\hat{A}, L).$$

The isogeny $\phi_L: \hat{A} \to A$ is not necessarily separable. Denote by $\hat{K}$ the kernel of $\phi_L$, let $\hat{K}_0$ be the maximal sub-group of $\hat{K}$ supported at $\hat{0}$ and let $A' = \hat{A}/\hat{K}_0$. Then we have a factorization
\begin{align*}
\phi_L: \hat{A} \xrightarrow{\nu} A' \xrightarrow{\mu} A
\end{align*}
where $\nu: \hat{A} \to A'$ is the natural quotient map which is purely inseparable, and $\mu: A'\to A$ is \'{e}tale.
Fix a sufficiently large integer $g_1$ such that $F_{A'}^{g_1}: A'^{g_1} \to A'$ factors through $\nu$. Let
$$\phi = \mu\circ F_{A'}^{g_1}: A'^{g_1} \to A' \to A.$$
Then $\phi^*\hat{L}^* \cong \bigoplus^r L'$ where $L'$ is an ample line bundle on $A'^{g_1}$.  For  a larger integer $e>{g_1}$, we get the following commutative diagram
$$\centerline{\xymatrix@C=1.3cm{
&X'^e \ar[dd]_{f'^e}\ar[rd]^{F_{X'}^e}\ar@/^2pc/[rr]^{h^e}\ar[r]^>>>>>>{\sigma}  &X_1'^e = X^e\times_A A'\ar[d]^{W_{X'}^e}\ar[r]^<<<<<<{h_1^e}  &X^e\ar[d]^{F_X^e}   \\
& &X' =X\times_A A' \ar[r]^<<<<<<{h}\ar[d]^<<<<<<{f'}                          &X\ar[d]^f      \\
&A'^e \ar[r]^{F_{A'}^e}  &A' \ar[r]^{\mu}                                  &A
}}$$
where $W_{X'}^e, h_1^e, f', h$ denote the natural projections of the corresponding fiber products, and $\sigma: X'^e \to X_1'^e$ arises from the universal property of the fiber product. Since the base change $h: X' \to X$ is \'{e}tale, $\sigma$ is an isomorphism.

Assume $l \geq 2$, $q_1|l$. Let $n_e = \llcorner\frac{1-p^e + lp^e}{q_1} \lrcorner$. Then we can write that $1-p^e + lp^e = n_e q_1 + r_eq_0$ where $0 \leq r_e <d$. It follows that
$$\O_{X}((1-p^e + lp^e)(K_X + B)) \cong \O_X(n_eH) \otimes \mathcal{G}_{r_e}.$$
By Theorem \ref{pf-OA} (3), to get that for every $i>0$,
$$H^i(A, f_*(F_{X*}^{e}\O_{X}((1-p^e)(K_X + B)) \otimes \O_X(l(K_X + B))) \otimes \hat{L}^*)=0,$$
we only need to verify that for every $i>0$,
$$H^i(A', \mu^*(f_*F_{X*}^{e}\O_X((1-p^e + lp^e)(K_X + B))) \otimes \hat{L}^*)) = 0.$$
This is true when $e\gg0$, which is proved as follows
\begin{equation*}
\small\begin{split}
&H^i(A', \mu^*(f_*F_{X*}^{e}\O_X((1-p^e + lp^e)(K_X + B))) \otimes \hat{L}^*)) \\
&\cong H^i(A', \mu^*(f_*F_{X*}^{e}\O_X((1-p^e + lp^e)(K_X + B))) \otimes \mu^*\hat{L}^*)) \\
&\cong H^i(A', f'_*W_{X'*}^{e}(h_1^{e*}\O_X((1-p^e + lp^e)(K_X + B))) \otimes \mu^*\hat{L}^*) \\
&\cong H^i(A', f'_*W_{X'*}^{e}\sigma_*(\sigma^*h_1^{e*}\O_X((1-p^e + lp^e)(K_X + B))) \otimes \mu^*\hat{L}^*) \\
&\cong H^i(A', f'_*F_{X'*}^{e}(h^{e*}\O_X((1-p^e + lp^e)(K_X + B))) \otimes \mu^*\hat{L}^*) \\
&\cong H^i(A', F_{A'*}^{e}f'^e_*(h^{e*}\O_X((1-p^e + lp^e)(K_X + B))) \otimes \mu^*\hat{L}^*) \\
&\cong H^i(A'^e, f'^e_*(h^{e*}\O_X((1-p^e + lp^e)(K_X + B))) \otimes F_{A'}^{e*}\mu^*\hat{L}^*) \\
&\cong H^i(A'^e, f'^e_*((h^{e*}\O_X((1-p^e + lp^e)(K_X + B))) \otimes (f'^{e*}F_{A'}^{(e-g_1)*}F_{A'}^{g_1*}\mu^*\hat{L}^*)) \\
&\cong H^i(A', f'_*(\O_{X'}(h^{*}n_eH) \otimes h^*\mathcal{G}_{r_e} \otimes \bigoplus^r f'^*(L')^{p^{e-g_1}})) \\
&\cong \bigoplus^r H^i(X', \O_{X'}(h^{*}n_eH + p^{e-g_1} f'^{*}L') \otimes h^*\mathcal{G}_{r_e}) = 0
\end{split}
\end{equation*}
where
\begin{itemize}
\item
the $2^{\mathrm{nd}}$ $\cong$ is due to $\mu^*f_*F_{X*}^{e} \cong f'_*W_{X'*}^{e}h_1^{e*}$ since $\mu$ is a flat base change;
\item
the $3^{\mathrm{rd}}$ $\cong$ is due to the fact that $\sigma$ is an isomorphism;
\item
the $6^{\mathrm{th}}$ $\cong$ is from applying projection formula and $RF_{A'*}^{e} \cong F_{A'*}^{e}$;
\item
the $9^{\mathrm{th}}$ $\cong$ is from applying Lerray spectral sequence and relative Fujita vanishing (Lemma \ref{rel-fujita}) $R^jf'_*(\O_{X'}(h^{*}n_eH) \otimes h^*\mathcal{G}_{r_e})= 0$ for $j>0$ since $n_eH$ is sufficiently $f'$-ample if $e\gg0$, and the last vanishing follows from applying Fujita vanishing since $h^{*}n_eH + p^{e-g_1} f'^{*}L'$ is sufficiently ample if $e\gg0$.
\end{itemize}

\subsection{Proof of Theorem \ref{fib-to-curve-k=1}}
We may assume that $K_X + B$ is nef by working on a log minimal model of $(X,B)$ over $Y$ (Theorem \ref{rel-mmp} (3.4)). As $\kappa(X_{\ol\eta}, K_{X_{\ol\eta}} + B_{\ol\eta}) = 1$, there exist a log resolution $\sigma: W \to X$ and a fibration $h: W \to Z$ to a smooth projective surface $Z$, which is birational to the relative Iitaka fibration induced by $\sigma^*(K_X +B)$ on $W$ over $Y$. These varieties fit into the following commutative diagram
$$\centerline{\xymatrix{ &W\ar[d]_h\ar[r]^{\sigma} &X\ar[d]^f  \\ &Z\ar[r]^g  &Y}}$$
By the assumption $\spadesuit$ and Proposition \ref{fib-pa=1}, the geometric generic fiber of $h$ is a smooth elliptic curve over $\ol{K(Z)}$. Applying flattening trick (\cite[Lemma 5.6]{BW14}), we can assume $\sigma^*(K_X + B) \sim_{\mathbb{Q}} h^*C$ where $C$ is a nef and $g$-big divisor on $Z$.
By \cite[Claim 3.1 and 3.2]{CZ15}, there exists an effective divisor $E$ on $W$ such that $K_W \sim_{\Q} h^*K_Z + E$.

For the case $g(Y) > 1$, applying Theorem \ref{thm-Zh16a} for $g: Z \to Y$, we show that for $n\gg0$, $nC + K_Z$ is big. By $K_W \sim_{\Q} h^*K_Z + E$, applying Theorem \ref{ct}, we have
\begin{align*}\kappa(X, K_X+ B) &= \kappa(X, (n+1)(K_X+ B)) \geq \kappa(X, n(K_X+ B) + K_X) \\
&=\kappa(W, \sigma^*n(K_X+ B) + K_W) \geq \kappa(Z, nC + K_Z) =2.
\end{align*}

Let's restrict on the case $g(Y) = 1$. We aim to prove that $\kappa(X, K_X + B) \geq 1$. First applying \cite[Theorem 1.2 and 1.3]{CZ15}, we have
$$\kappa(X, K_X +B) \geq \kappa(X) = \kappa(W) \geq \kappa(Z) \geq \kappa(Z, K_{Z_{\bar{\eta}}}).$$
So we may assume that $\kappa(Z, K_{Z_{\bar{\eta}}}) \leq 0$, i.e., $p_a(Z_{\bar{\eta}}) \leq 1$, hence the geometric generic fiber $Z_{\bar{\eta}}$ is either a smooth elliptic curve (Proposition \ref{fib-pa=1}) or a rational curve over $k(\bar{\eta})$.
And if $\nu(Z, C) = 2$ then $C$ is big, applying Theorem \ref{ct} we are done by
$$\kappa(X, K_X + B) = \kappa(W, \sigma^*(K_X + B)) = \kappa(Z, C) =2.$$
So we may assume additionally that $\nu(Z, C) =1$.

We will mimic the proof of Theorem \ref{fib-to-ab-var} and break the arguments into three steps for similar purposes.
\smallskip

Step 1:
By the assumptions above, $g: Z \to Y$ is generically smooth. And since $K_Y \sim 0$ and $C$ is nef and $g$-big, we can apply Lemma \ref{positivity-surf} and obtain that, for sufficiently divisible $n>0$, $V': = g_*\O_Z(nC + K_{Z})$ is a nef vector bundle of rank $\geq 2$.
Fix a sufficiently divisible $N>0$ such that $NK_W \sim h^*NK_Z + NE$ where $NE$ is integral, and that $N(K_X + B)$ is Cartier.
Applying the projection formula, we can get a natural inclusion $\mathcal{O}_Z(NK_{Z}) \hookrightarrow h_*\mathcal{O}_W(NK_{W})$. In turn we have
\begin{align*}
&Sym^N V' = Sym^N g_*\O_Z(nC + K_{Z}) \to g_*\O_Z(nNC + NK_{Z}) \\
&\hookrightarrow g_*h_*\O_W(nN\sigma^*(K_X + B) + NK_{W}) \cong f_*\sigma_*\O_W(nN\sigma^*(K_X + B) + NK_{W})\\
&\subseteq f_*\O_X(Nn(K_X + B) + NK_{X}) \subseteq f_*\O_X(Nn(K_X + B) + N(K_{X} + B))\\
&= f_*\O_X(N(n+1)(K_X + B)).
\end{align*}
Denote by $V$ the image of $Sym^N V'$ via the above composition map. Let $l = N(n+1)$. Then $V$ is a nef sub-vector bundle of $f_*\O_X(l(K_X + B))$, and $\mathrm{rk}~V = r' \geq 2$.


By \cite[Proposition 2.7]{EZ16}, there exists a flat base change $\pi: Y_1 \to Y$ between elliptic curves such that $\pi^*V \cong \bigoplus_{i=1}^{i=r'} L_i'$ where each $L'_i$ is a line bundle with $\deg L'_i \geq 0$. Consider the following commutative diagram
\begin{align}\label{comm-diag}
\xymatrix@C=1cm{ &X_1'\ar[rd]_{f_1'}\ar[r]_<<<<<<{\nu}\ar@/^1pc/[rr]^{\pi_1'} &X_1 = X\times_Y Y_1 \ar[d]_{f_1} \ar[r]_<<<<{\pi_1} &X\ar[d]^{f}\\
                         &     &Y_1\ar[r]_{\pi}                                 &Y
}
\end{align}
where $X_1'$ is the normalization of $X_1$. We have a natural inclusion map
\begin{align*}
\alpha: \bigoplus_{i=1}^{i=r'} L'_i \cong &\pi^*V \subseteq \pi^*(f_*\O_X(l(K_X + B)))\\
 &\cong  f_{1*}\O_{X_1}(\pi_1^*(l(K_X + B)))   \subseteq f_{1*}'\O_{X_1'}(\pi_1'^*(l(K_X + B))).
\end{align*}

If some $L'_{i_0}$ satisfies that $\deg L'_{i_0} >0$, by doing a further \'etale base change, we may assume $\deg L'_{i_0} \geq 2$, thus
$$h^0(X_1', \pi_1'^*(m(K_X + B))) = h^0(Y_1, f_{1*}'\O_{X_1'}(\pi_1'^*m(K_X + B))) \geq  h^0(Y_1, L'_{i_0}) \geq 2,$$
which implies $\kappa(X, K_X + B) = \kappa(X', \pi_1'^*(K_X + B)) \geq 1$ by Theorem \ref{ct}.

From now on, we assume every $L'_i \in \mathrm{Pic}^0(Y_1)$.

\smallskip

Step 2: Granted the inclusion map $\alpha$ and the commutative diagram (\ref{comm-diag}) in Step 1, as in Step 2 of the proof of Theorem \ref{fib-to-ab-var},
we can find an integer $m_1 = l_1l$ and some divisors $D_i \in |m_1(K_X + B) + f^*L_i|, i=1,2,\cdots, r$ for some $L_i \in \mathrm{Pic}^0(Y)$, such that the sub-linear system of $|m_1(K_X + B)_{\eta}|$ generated by $(D_i)_{\eta}, i=1,2,\cdots, r$ defines a nontrivial map of $X_{\eta}$.
\smallskip

We only need to prove that there exist at least two different divisors among $D_i$, say, $D_1 \neq D_2$, such that $L_1, L_2$ are torsion in $\mathrm{Pic}^0(Y)$ (Step 2 of Theorem \ref{fib-to-ab-var}).
\smallskip

Step 3: We will construct a minimal dlt pair $(\hat{X}, \hat{B})$ and divisors $\hat{D}_1, \hat{D}_2$ satisfying the conditions of Lemma \ref{tor-lem}.

(3.1) Take a log resolution $\mu: \tilde{X} \rightarrow X$ of the pair $(X, B + \sum_iD_i)$. Denote by $\tilde{f}: \tilde{X} \rightarrow Y$ the natural morphism.
Let $\tilde{B}$ be the reduced divisor supported on the union of $\sum_i \mu^*D_i$ and the exceptional divisors.
By running a log MMP, we get a minimal dlt model $(\hat{X}, \hat{B})$, and a natural morphism $\hat{f}: \hat{X} \to Y$. The divisor $\tilde{E} = K_{\tilde{X}} + \tilde{B} - \mu^*(K_X + B)$ is effective. Take a sufficiently divisible integer $l_2>0$ such that $l_2\tilde{E}$ is Cartier. Let $m_2 = m_1l_2$. We get effective divisors
$$\tilde{D}_i = l_2\mu^*D_i + l_2\tilde{E} \sim m_2(K_{\tilde{X}} + \tilde{B}) +  l_2\mu^*f^*L_i$$
and the push-forward divisors via the natural map $\tilde{X} \dashrightarrow \hat{X}$
$$\hat{D}_i \sim m_2(K_{\hat{X}} + \hat{B}) +  l_2\hat{f}^*L_i.$$

(3.2) We can prove $\nu(K_{\hat{X}} + \hat{B}) = \nu(K_X +B)=1$ as in Step 3 (3.2) of the proof of Theorem \ref{fib-to-ab-var}.
Note that $(K_{\hat{X}} + \hat{B})_{\eta}$ is semi-ample by Theorem \ref{rel-mmp} (3.2), thus $\kappa(\hat{X}_{\eta}, (K_{\hat{X}} + \hat{B})_{\eta}) =1$. Considering the relative Iitaka fibration $\hat{h}':\hat{X} \dashrightarrow \hat{Z}'$ of $\hat{X}$ over $Y$  induced by $K_{\hat{X}} + \hat{B}$ and applying flattening trick (\cite[Lemma 5.6]{BW14}), we get a commutative diagram
$$\centerline{\xymatrix{
&\hat{W}\ar[d]_{\hat{h}}\ar[r]^{\hat{\sigma}} &\hat{X}\ar[d]^{\hat{f}} &\tilde{X}\ar@{.>}[l]\ar[ld]^{\tilde{f}} \\
&\hat{Z} \ar[r]^{\hat{g}} &Y  &
}}$$
such that $\hat{\sigma}^*(K_{\hat{X}} + \hat{B}) \sim_{\mathbb{Q}} \hat{h}^*\hat{C}$ for some nef and $\hat{g}$-big divisor $\hat{C}$ on $\hat{Z}$ , where $\hat{W}$ and $\hat{Z}$ are smooth, birational to $\hat X$ and $\hat Z'$ respectively, and $\hat{h}$ is a fibration birational to $\hat{h}'$. By the construction, $\hat{h}: \hat{W} \to \hat{Z}$ is an elliptic fibration, and
$\nu(\hat{Z}, \hat{C}) = 1$.


(3.3) We can take effective divisors $\hat{C}_i \sim m_2\hat{C} + l_2\hat{g}^*L_i$ on $\hat{Z}$ such that $\hat{h}^*\hat{C}_i = \hat{\sigma}^*\hat{D}_i$.
Note that $\hat{C}_i$ is nef and $\hat{C}_i^2 =0$.
Considering the connected components of the union of $\hat C_i, 0 \leq i \leq r$, we can show that there exist effective Cartier divisors $\hat C_1',\ldots,\hat C_s'$ satisfying the following conditions:
\begin{itemize}
\item $\mathrm{Supp}(\hat{C}_j')$ is connected for each $j$, and $\mathrm{Supp}(\hat{C}_{j_1}')\cap\mathrm{Supp}(\hat{C}_{j_2}')=\emptyset$ if $j_1\neq j_2$;
\item every $\hat C_j'$ is nef and $(\hat{C}_j')^2=0$;
\item the greatest common divisor of the coefficients of every $\hat{C}_j'$ is equal to one;
\item for each $i$, there exist $a_{i1},\ldots,a_{is} \in \mathbb{Z}^{\geq 0}$ such that $\hat{C}_i = a_{i1}\hat{C}_1' + \cdots + a_{is}\hat{C}_s'$.
\end{itemize}
As $\hat{\sigma}^*\hat{D}_i = \hat{h}^*\hat{C}_i$, by the construction,
$\hat h^*\hat C_1',\ldots,\hat h^*\hat C_s'$ are disjoint connected components of $\hat \sigma^*(\sum \hat D_i)$.
Let $\hat{G}_j:=\hat\sigma_*\hat h^*\hat C_j'$. Then $\mathrm{Supp}(\hat{G}_{j_1})\cap\mathrm{Supp}(\hat{G}_{j_2})=\emptyset$ if $j_1\neq j_2$, and
$\hat{D}_i = a_{i1}\hat{G}_1 + \cdots + a_{is}\hat{G}_s$.

(3.4) We can take two divisors among $\hat{D}_i$, say, $\hat{D}_1,\hat{D}_2$ such that $(\hat{D}_1)_\eta \neq (\hat{D}_2)_\eta$. And since $(\hat{D}_1)_\eta \sim (\hat{D}_2)_\eta$, we can find two connected components among $\hat{G}_j$, say, $\hat{G}_1, \hat{G}_2$, such that $(\hat{G}_1)_\eta, (\hat{G}_2)_\eta >0$ and $a_{11}>a_{21}\geq0$ and $a_{22}>a_{12} \geq 0$.
And if we write that
$$\hat D_1 = a_{11}\hat{G}_1 + a_{12}\hat{G}_2 + \hat{G}'_1~\mathrm{and}~ \hat D_2 = a_{21}\hat{G}_1 + a_{22}\hat{G}_2 + \hat{G}'_2$$
then for $i,j \in \{1,2\}$ and $i \neq j$, $\hat{G}_j$ does not intersect $\hat{G}''_j:=\hat{G}_i + \hat{G}'_1 + \hat{G}'_2$.

(3.5) Take two reduced, irreducible and dominant over $Y$ components $G_1, G_2$ of $\hat{G}_1, \hat{G}_2$ respectively.
Since $K_{\hat{X}} + \hat{B}$ is nef and $\nu(\hat{X}, \hat{D}_j)=\nu(\hat{X}, K_{\hat{X}} + \hat{B}) = 1$, applying Proposition \ref{num-prop}
we obtain that for $j =1,2$,
$$\nu(G_j, (K_{\hat{X}} + \hat{B})|_{G_j}) = \nu(G_j, \hat{D}_j|_{G_j}) \leq \nu(\hat{X}, \hat{D}_j)- 1=0,$$
hence $(K_{\hat{X}} + \hat{B})|_{G_j} \sim_{\Q} 0$ (Lemma \ref{tor-lem}), so the Iitaka fiber $F_j=G_j$.

Finally, as the conditions of Lemma \ref{tor-lem} are verified for the pair $(\hat{X}, \hat{B})$ and the divisors $\hat{D}_1,\hat{D}_2$, we conclude that $L_1$ and $L_2$ are torsion. The proof is completed.

\subsection{Proof of Theorem \ref{fib-to-curve}}
We can assume $(X, B)$ is relatively minimal over $Y$. Then by Theorem \ref{rel-mmp} (3.3, 3.4), $K_X + B$ is nef and $f$-$\Q$-trivial, so we can assume $l(K_X + B) \sim f^*L$ for some integer $l>0$ and a nef line bundle $L$ on $Y$. If $\deg L >0$ then $\kappa(X, K_X + B) \geq 1 \geq \kappa(Y)$.
So we may assume $\deg L = 0$. And we only need to prove that $g(Y) =1$ and $L$ is torsion in $\mathrm{Pic}^0(Y)$.

\begin{lem}\label{pd-nef}
Let $D$ be a pseudo-effective $\Q$-Cartier $\Q$-divisor on $X$. If $D$ is $f$-nef, then $D$ is nef.
\end{lem}
\begin{proof}
Assume the contrary. We can take an ample $\Q$-divisor $H$ on $X$ such that $D+H$ is not nef. Since $D$ is pseudo-effective, there exists an effective
$\Q$-divisor $B'\sim_{\Q} D + H$. Take a small rational number $t > 0$ such that $(X, B + tB')$ is klt. Since $K_X + B + tB'$ is $f$-nef, applying Theorem \ref{rel-mmp} shows that $tB' \equiv K_X + B + tB'$ is nef, which contradicts the assumption.
\end{proof}

Let $H$ be an ample Cartier divisor on $X$ and $F$ a general fiber of $f$. We can find two positive integers $a,b$ such that $D^3 = 0$ where $D = aH - bF$.

\medskip

\emph{We claim that $D$ is nef, thus $\nu(D) =2$.}  By Lemma \ref{pd-nef} it suffices to show that $D$ is pseudo-effective. Otherwise, denote by $t_0$ the maximal real number such that $H - t_0F$ is pseudo-effective. Then $t_0 < \frac{b}{a}$ and $(H - t_0F)^3 >0$. Since for every rational number $t<t_0$ the divisor $H - tF$ is $f$-ample and pseudo-effective, so $H - tF$ is nef by Lemma \ref{pd-nef}. Taking the limit shows that $H-t_0F$ is nef, thus $H-t_0F$ is big since $(H - t_0F)^3 >0$ (\cite[Chap. V Lemma 2.1]{Nak04}).\footnote{this is well known for $\Q$-Cartier $\Q$-divisors (\cite[Prop. 2.61]{KM98}), but the proof also applies for $\R$-Cartier $\R$-divisors as mentioned by \cite[Chap. V Lemma 2.1]{Nak04}} It follows that for a sufficiently small $\epsilon >0$ such that $t_0 + \epsilon \in \Q$, the divisor $H - t_0F - \epsilon F$ is $\Q$-linearly equivalent to an effective divisor. However, this contradicts the definition of $t_0$.

\medskip

Take a smooth resolution of singularities $\sigma: W \to X$ and let $h = f \circ\sigma: W \to X \to Y$. Since $n\sigma^*D + K_{W/Y} - K_{W/Y} = n\sigma^*D$ is nef, $h$-big and $h$-semi-ample, applying Theorem \ref{thm-Zh16a} shows that for sufficiently divisible integers $n$ and $g$, the sheaf $F_Y^{g*}h_*\O_W(n\sigma^*D + K_{W/Y})$ contains a nonzero nef subbundle $V$.

\smallskip

We exclude the case $g(Y) > 1$ as follows. Otherwise, for sufficiently large $n$, the divisor $n\sigma^*D + K_W$ is big by Theorem \ref{cor-Zh16a}. We can find an effective $\sigma$-exceptional divisor $E$ on $W$ such that $K_W \leq \sigma^*K_X + E$. Then $n\sigma^*D + \sigma^*(K_X + B) + E  \geq n\sigma^*D + K_W + \sigma^*B$ is big. Applying Theorem \ref{ct},
$$\kappa(X, nD + (K_X + B)) = \kappa(W, n\sigma^*D + \sigma^*(K_X + B) + E) = 3.$$
As $K_X + B$ is numerically trivial, we conclude that $D$ is big, but this contradicts that $\nu(D) = 2$.

\smallskip

We may assume $Y$ is an elliptic curve. Next we will prove that there exists a semi-ample divisor $D' \equiv tD$ for some rational number $t>0$.
Since the subbundle $V\subseteq F_Y^{g*}h_*\O_W(n\sigma^*D + K_{W/Y})$ is nef, for every $L'' \in \mathrm{Pic}^0(Y)$, applying Riemann-Roch formula gives
$$h^0(Y, V\otimes L'') - h^1(Y, V\otimes L'') = \chi(Y, V\otimes L'') = \deg V \geq 0.$$
We claim that \emph{there exists some $L'' \in \mathrm{Pic}^0(Y)$ such that $h^0(Y, V\otimes L'') >0$}. Otherwise, for every $L'' \in \mathrm{Pic}^0(Y)$ and $i=0,1$, $h^i(Y, V\otimes L'') =0$, thus $R^i\Phi_{\mathcal{P}}V = 0, i=0,1$, i.e.,
the Fourier-Mukai transform $R\Phi_{\mathcal{P}}V$ is acyclic (\cite[p.53]{Har66}). Therefore $V=0$ by Theorem \ref{Mu}, and the claim follows from this contradiction.  Note that since $\dim Y =1$, by Proposition \ref{sep-fib} the fibration $h: W\to Y$ is separable, hence $W_{Y^{g}} = W\times_Y Y^g$ is integral.
Consider the following commutative diagram \\
\[\xymatrix@C=2cm{&W'\ar@/^1.5pc/[rr]|{\rho}\ar[r]^{\rho'}\ar[dr]^{h'} &W_{Y^{g}}\ar[r]^{\pi^{g}}\ar[d]^{h_{g}}     &W\ar[d]^h\\
&      &Y^{g}\ar[r]^{F_{Y}^{g}}     &Y\\
} \]
where $W'$ denotes the normalization of $W_{Y^{g}}$ and $\rho, \rho', h'$ denote the natural morphisms.
From the commutative diagram above, by \cite[Chapter III, Prop. 9.3]{Har77} we obtain
\begin{align*}
F_{Y}^{g*}h_*\mathcal{O}_W(n\sigma^*D + K_W)& \cong h_{g*}(\pi^{g*} \mathcal{O}_W(n\sigma^*D + K_W)) \cong h_{g*} \mathcal{O}_{W_{Y^g}}(\pi^{g*}(n\sigma^*D + K_W))  \\
&\hookrightarrow h'_*\mathcal{O}_{W'}(\rho'^* \pi^{g*}(n\sigma^*D + K_W)) = h'_*\mathcal{O}_{W'}(\rho^*(n\sigma^*D + K_W)).
\end{align*}
It follows that
\begin{align*}
h^0(\mathcal{O}_{W'}(\rho^*(n\sigma^*D + K_W))\otimes h^*L'') &=  h^0(Y, h'_*\mathcal{O}_{W'}(\rho^*(n\sigma^*D + K_W))\otimes L'') \\
& \geq h^0(Y, V\otimes L'') >0.
\end{align*}
Take $L' \in \mathrm{Pic}^0(Y)$ such that $F_Y^{g*}L' \sim L''$. Then applying Theorem \ref{ct}, we can show
\begin{align*}
\kappa(X, nD + (K_X + B) + f^*L') &\geq \kappa(W, n\sigma^*D + K_W+ h^*L') \\
&= \kappa(W',\rho^*(n\sigma^*D + K_W) + h'^*L'')  \geq 0.
\end{align*}
There exists an effective divisor $D_1 \sim_{\Q} nD + (K_X + B) + f^*L'$.
The pair $(X, B + tD_1)$ is klt for a sufficiently small rational number $t$, so the divisor $K_X + B + tD_1$ is semi-ample by Theorem \ref{fib-to-ab-var}.
Since $K_X + B \equiv 0$, we can set $D' = K_X + B + tD_1$.

\smallskip

Since $D'$ is $f$-ample and $\nu(D') = 2$, the associated morphism to $D'$ is a fibration $g: X \to Z$ to a surface. Let $X_{\xi}$ be the generic fiber of $g$. Then $X_{\xi}$ is a normal curve defined over $k(\xi)=K(Z)$ and dominant over $Y\otimes_k k(\xi)$. It follows that $p_a(X_{\xi}) \geq 1$, thus $K_{X_{\xi}}$ is semi-ample. Since $l(K_X + B)|_{X_{\xi}} \sim_{\Q} l(K_{X_{\xi}} + B|_{X_{\xi}})$ is numerically trivial, we conclude that $K_{X_{\xi}}  \sim_{\Q} K_X|_{X_{\xi}} \sim_{\Q} 0$ and $B|_{X_{\xi}}  = 0$, thus $l(K_X + B)|_{X_{\xi}} \sim_{\Q} f^*L|_{X_{\xi}} \sim_{\Q} 0$. Though the geometric generic fiber $X_{\bar{\xi}}$ is not necessarily reduced, we can apply Lemma \ref{inj-pic} to the morphism $(X_{\bar{\xi}})_{\mathrm{red}} \to Y\otimes_k k(\bar{\xi})$ and show that $L$ is torsion.

\subsection{Proof of Theorem \ref{Iit-conj}}

In the following we assume that $\kappa(X_{\ol\eta}, K_{X_{\ol\eta}} + B_{\ol\eta}) \geq 0$.

\smallskip

\textbf{Case (i) $\dim Y = 1$.}
If $\kappa(X_{\ol\eta}, K_{X_{\ol\eta}} + B_{\ol\eta}) = 0$ then we can use Theorem \ref{fib-to-curve}. If $\kappa(X_{\ol\eta}, K_{X_{\ol\eta}} + B_{\ol\eta}) = 1$ then we can use Theorem \ref{fib-to-curve-k=1}. If $\kappa(X_{\ol\eta}, K_{X_{\ol\eta}} + B_{\ol\eta}) = 2$, then we can use Theorem \ref{fib-to-ab-var} when $g(Y) = 1$ and use Corollary \ref{cor-Zh16a} when $g(Y) > 1$.




\smallskip

\textbf{Case (ii) $\dim Y = 2$.}
If $\kappa(X_{\ol\eta}, K_{X_{\ol\eta}} + B_{\ol\eta}) = 0$, then by the assumption $\spadesuit$, $f$ is an elliptic fibration (Proposition \ref{fib-pa=1}), applying \cite[Theorem 1.2]{CZ15} shows that
$$\kappa(X, K_X + B) \geq\kappa(X) \geq \kappa(Y).$$
So we may assume $K_X +B$ is $f$-big. We fall into three cases according to $\kappa(Y)$.
\smallskip

When $\kappa(Y) = 2$, we can use Corollary \ref{cor-Zh16a}.
\smallskip

When $\kappa(Y) = 1$, we need to prove that $\kappa(X, K_X + B) \geq 2$. Denote by $g: Y\to C$ the Iitaka fibration of $Y$ and by $\mathrm{alb}_{Y}: Y \to \mathrm{Alb}(Y)$ the Albanese map which is generically finite. Take a general fiber $G$ of $g$, which is an elliptic curve. Let $E=\mathrm{alb}_Y(G)$. We may assume $\mathrm{Alb}(Y)$ is an abelian variety and $E$ passes through the origin. Let $A(E)$ be the sub-abelian variety generated by $E$. It follows that $1\leq \dim A(E) \leq g(E) \leq g(G) =1$, hence the equalities are attained, and $E= A(E)$ is a sub-abelian variety.
Since the composition morphism $\phi: Y \to \mathrm{Alb}(Y) \to \mathrm{Alb}(Y)/E$ contracts the fiber $G$, applying Rigidity Lemma (\cite[Proposition 6.1]{GIT}), every fiber of $g$ is contracted by $\phi$, and $\phi: Y \to \mathrm{Alb}(Y)/E$ factors through $g: Y \to C$.
By the universal property
of the Albanese map we have the following commutative diagram
\[\xymatrix@C=1.7cm{&Y\ar[d]^g\ar[r]^{\mathrm{alb}_{Y}}  &\mathrm{Alb}(Y)\ar[d]^{h}\ar[dr]^{\phi} &\\
&C\ar[r]^{\mathrm{alb}_{C}} \ar[r]  &\mathrm{Alb}(C)\ar[r] &\mathrm{Alb}(Y)/E
}.\]
Wee see that the morphism $\mathrm{Alb}(C) \to \mathrm{Alb}(Y)/E$ is in fact an isomorphism, hence $h: \mathrm{Alb}(Y) \to \mathrm{Alb}(C)$ is of relative dimension one, in particular $g(C) \geq 1$.
Next let $Y \xrightarrow{\sigma} \bar{Y} \to \mathrm{Alb}(Y)$ be the Stein factorization. Then $\sigma: Y \to \bar{Y}$ is a birational morphism, and the $\sigma$-exceptional locus does not intersect the generic fiber of $g: Y \to C$. Therefore, the natural morphism $\bar{Y} \to \mathrm{Alb}(C)$ factors through a morphism $\bar{g}: \bar{Y} \to C$, which is induced by the Stein factorization.
Let $(\bar{X}, \bar{B})$ be a log minimal model of $(X,B)$ over $\bar{Y}$. By Theorem \ref{rel-mmp} (3.4), $(\bar{X}, \bar{B})$ is in fact minimal. Then we have the following commutative diagram
\[\xymatrix@C=1.7cm{&X\ar@{.>}[r]  &\bar{X}\ar[d]_{\bar{f}'}\ar[r]  &\bar{Y}\ar[r]\ar[d]^{\bar{g}}     &\mathrm{Alb}(Y)\ar[d]^h\\
& &C' \ar[r]  &C \ar[r]^{\mathrm{alb}_Y} &\mathrm{Alb}(C)
} \]
where $\bar{f}': \bar{X} \to C'$ is the fibration induced by the Stein factorization of $\bar{X} \to C$. Note that $g(C') \geq g(C) \geq 1$.
Denote by $\zeta'$ the generic point of $C'$. Consider the fibration $\bar{f}': \bar{X} \to C'$. By Theorem \ref{rel-mmp} (3.2), $(K_{\bar{X}} +\bar{B})_{\zeta'}$ is semi-ample. And since $K_{\bar{X}} +\bar{B}$ is big over $\mathrm{Alb}(Y)$, it is not numerically trivial over the generic point of $C'$, which implies that $\kappa(\bar{X}_{\zeta'}, (K_{\bar{X}} +\bar{B})_{\zeta'}) \geq 1$.
\smallskip

\begin{cln}\label{cln}
If $\kappa(\bar{X}_{\zeta'}, (K_{\bar{X}} +\bar{B})_{\zeta'}) = 1$ then $\bar{f}'$ satisfies the assumption $\spadesuit$.
\end{cln}
\emph{Proof.}
Denote by $\varphi: \bar{X} \dashrightarrow Z$ the relative Iitaka fibration induced by $K_{\bar{X}} +\bar{B}$ on $\bar{X}$ over $C'$ and by $\bar{X}_{\xi}$ the generic fiber of $\varphi$. Then $\bar{X}_{\xi}$ is a normal curve over $K(Z)$. As $K_{\bar{X}} +\bar{B}$ is big over $\mathrm{Alb}(Y)$ while $(K_{\bar{X}} +\bar{B})|_{\bar{X}_{\xi}} \sim_{\Q} 0$, we see that $\bar{X}_{\xi}$ is not contracted by $\bar{X} \to \mathrm{Alb}(Y)$, hence is generically finite over $\mathrm{Alb}(Y)\otimes_kK(Z)$. Therefore, $p_a(\bar{X}_{\xi}) \geq 1$, and $K_{\bar{X}_{\xi}}$ is semi-ample. Since $(K_{\bar{X}} +\bar{B})|_{\bar{X}_{\xi}} \sim_{\Q} K_{\bar{X}_{\xi}} + \bar{B}|_{\bar{X}_{\xi}}$ is numerically trivial,
we conclude that $\bar{B}|_{\bar{X}_{\xi}} = 0$, thus $\bar{f}'$ satisfies the assumption $\spadesuit$. \qed \\
So we can apply the results of Case (i) to the fibration $\bar{f}': \bar{X} \to C'$ and show that
$$\kappa(\bar{X}, K_{\bar{X}} +\bar{B}) \geq \kappa(\bar{X}_{\zeta'}, (K_{\bar{X}} +\bar{B})_{\zeta'}) + \kappa(C') \geq 1.$$
Hence $K_{\bar{X}} + \bar{B}$ is semi-ample by Theorem \ref{s-amp-k=2}. And by Corollary \ref{cor-Zh16a}, we have $\nu(\bar{X}, K_{\bar{X}} +\bar{B}) \geq 2$, hence $\kappa(\bar{X}, K_{\bar{X}} +\bar{B}) \geq 2$.



\smallskip

When $\kappa(Y) = 0$, we need to prove that $\kappa(X, K_X + B) \geq 1$. In this case $Y$ is birationally equivalent to an abelian surface (\cite[Sec. 10]{Ba01}). We may assume $Y$ is an abelian surface and assume $(X,B)$ is minimal by working on a log minimal model over $Y$. If $Y$ is simple then we can use Theorem \ref{fib-to-ab-var}. Otherwise, $Y$ admits a fibration to an elliptic curve, then we get a fibration $f': X \to C'$ with $g(C') =1$. Let $\zeta'$ denote the generic point of $C'$. Again since $K_X + B$ is big over $Y$, we have $\kappa(X_{\zeta'}, K_{X_{\zeta'}} + B_{\zeta'}) \geq 1$, and if $\kappa(X_{\zeta'}, K_{X_{\zeta'}} + B_{\zeta'}) = 1$
we can verify that $f'$ satisfies the assumption $\spadesuit$ by the same proof of Claim \ref{cln}. Applying the results of Case (i) to the fibration $f': X \to C'$, we complete the proof by
$$\kappa(X, K_X + B) \geq \kappa(X_{\zeta'}, K_{X_{\zeta'}} + B_{\zeta'}) + \kappa(C') \geq 1.$$

\section{Abundance}\label{sec-pf-abundance}
In this section, we will prove Theorem \ref{abundance-log}.
By Theorem \ref{s-amp-k=2}, we only need to show that either $\kappa(X, K_X + B) \geq 1$ or $K_X + B\sim_{\Q} 0$. Hence we may assume $\kappa(X, K_X + B) \leq 0$. If $X$ is of maximal Albanese dimension, then we are done by \cite[Theorem 1.1]{Zh16c} or \cite[Theorem 0.3]{HPZ17}. Let $f:X\to Y$ be the fibration arising from the Stein factorization of $a_X$. Then $\kappa(X_{\eta}, (K_X + B)_{\eta}) \geq 0$ by Theorem \ref{rel-mmp} (3.2). Therefore, by Theorem \ref{Iit-conj} it is only possible that
$\kappa(Y) = \kappa(X_{\eta}, (K_X + B)_{\eta}) = 0$. If $\dim Y = 1$ then we can use Theorem \ref{fib-to-curve}. If $\dim Y = 2$ then $B=0$ by the assumption (1), and $f$ is an elliptic fibration by Proposition \ref{fib-pa=1}, so $X$ is non-uniruled, and this case has been treated in \cite[Theorem 1.1]{Zh16c}.


\begin{thebibliography}{}

\bibitem{Ba01} L. B\v{a}descu, \textit{Algebraic surfaces}, Universitext, Springer-Verlag, New York, 2001.


\bibitem{Bir16} C. Birkar, \textit{Existence of flips and minimal models for 3-folds in char p}, Ann. Sci. \'{E}c. Norm. Sup\'{e}r. \textbf{49} (2016), No. 4, 169--212.

\bibitem{BCZ15} C. Birkar, Y. Chen and L. Zhang, \textit{Iitaka's $C_{n,m}$ conjecture for $3$-folds over finite fields}, Nagoya Math. J. \tb{229} (2018), 21--51.

\bibitem{BW14} C. Birkar and J. Waldlon, \textit{Existence of Mori fibre spaces for 3-folds in char $p$}, Adv. Math. \tb{313} (2017), 62--101.

\bibitem{CHMS14} P. Cascini, P., C. Hacon, M.  Musta\c{t}\v{a} and K. Schwede: {\it
On the numerical dimension of pseudo-effective divisors in positive characteristic},
Amer. J. Math. \tb{136} (2014), no. 6, 1609--1628.

\bibitem{CP08} V. Cossart, O. Piltant, \textit{Resolution of singularities of threefolds in
positive characteristic I}, J. Algebra \textbf{320} (2008), 1051--1082.

\bibitem{CTX15} P. Cascini, H. Tanaka, C. Xu, {\it On base point freeness in positive characteristic}, Ann. Sci. \'{E}c. Norm. Sup\'{e}r. \tb{48} (2015), no. 5, 1239--1272.

\bibitem{CZ15} Y. Chen and L. Zhang, \textit{The subadditivity of the Kodaira Dimension for Fibrations of Relative Dimension One in Positive Characteristics}, Math. Res. Lett. \tb{22} (2015), 675--696.

\bibitem{Deb01} O. Debarre, \textit{Higher dimensional algebraic geometry}, Universitext, Springer-Verlag, New York, 2001.

\bibitem{DS15} O. Das and K. Schwede, \textit{The $F$-different and a canonical bundle formula}, Ann. Sc. Norm. Super. Pisa Cl. Sci. (5) \tb{17} (2017), no. 3, 1173--1205.

\bibitem{Ej15} S. Ejiri, \textit{Weak positivity theorem and Frobenius stable canonical rings of geometric generic fibers}, J. Algebraic Geom. \textbf{26} (2017), 691--734.

\bibitem{EZ16} S. Ejiri and L. Zhang, \textit{Iitaka's $C_{n,m}$ conjecture for 3-folds in positive characteristic}, Math. Res. Lett. \textbf{25} (2018), 783--802.

\bibitem{Ek87} T. Ekedahl, \textit{Foliation and inseparable morphism}, Proceedings of Symposia in Pure Mathematics, V. \textbf{46} (1987), 139--149.

\bibitem{Ga} O. Gabber, \textit{Notes on some $t$-structures}, Geometric aspects of Dwork theory. Vol. I, II, Walter de Gruyter GmbH and Co. KG, Berlin, 2004, pp. 711--734.

\bibitem{SGA1} A. Grothendieck, {\it Rev\'{e}tements \'{e}tales et groupe fondamental (SGA 1)}, Lect. Notes in Math. \textbf{224}, Springer-Verlag, Berlin, 1971.

\bibitem{Har66} R. Hartshorne, \textit{Residues and duality}, Lect. Notes in Math. \tb{20}, 1966.

\bibitem{Har77} R. Hartshorne, \textit{Algebraic geometry}, Springer-Verlag, New York, Graduate Texts inMathematics \tb{52}, 1977.

\bibitem{Hac04} C.D. Hacon, \textit{A derived category approach to generic vanishing}, J. Reine Angew. Math. \tb{575} (2004), 173--187.

\bibitem{HP16} C. D. Hacon, Z. Patakfalvi, {\it Generic vanishing in characteristic
$p>0$ and the characterization of ordinary abelian varieties.} Amer. J. Math. \tb{138} (2016), no. 4,
963--998.

\bibitem{HPZ17} C.D. Hacon, Z. Patakfalvi and L. Zhang, \textit{Birational characterization of abelian varieties and ordinary abelian varieties in characteristic $p > 0$}, Duke Math. J. \textbf{168} (2019), no. 9, 1723--1736.

\bibitem{HX15} C. Hacon and C. Xu, \textit{On the three dimensional minimal model program in positive characteristic}, J. Amer. Math. Soc. \textbf{28} (2015), 711--744.

\bibitem{Iit82} S. Iitaka, \textit{Algebraic geometry}, An introduction to birational geometry of algebraic varieties,
Graduate Texts in Mathematics, \textbf{76}, 1982.

\bibitem{J97} A. J. de Jong, \textit{Families of curves and alterations}. Ann. Inst. Fourier (Grenoble) \textbf{47} (1997), no. 2, 599--621.

\bibitem{KM98} J. Koll\'{a}r and S. Mori, \textit{Birational geometry of algebraic varieties}, Cambridge Tracts in Mathematics \textbf{134}, 1998.

\bibitem{KMM94} S. Keel, K. Matsuki and J. McKernan, \textit{Log abundance theorem for threefolds}, Duke Math. J. \tb{75} (1994), no. 1, 99--119.

\bibitem{Ke03} D. S. Keeler, \textit{Ample filters of invertible sheaves}, J. Algebra \tb{259} (2003), 243--283.

\bibitem{LS77} H. Lange and U. Stuhler, {\it Vektorb\"undel auf Kurven und Darstellungen Fundamentalgruppe}, Math Z. \tb{156} (1977), 73--83.

\bibitem{Liu10} Q. Liu, \textit{Algebraic geometry and arithmetic curves},
Oxford University Press, 2010.

\bibitem{MP97} Y. Miyaoka and T. Peternell, {\it Geometry of higher-dimensional algebraic varieties}, DMV Seminar, \textbf{26}. Birkh\"{a}user Verlag, Basel, 1997.

\bibitem{Mu81} S. Mukai, \textit{Duality between $D(X)$ and $D(\hat{X})$ with its application to Picard sheaves}, Nagoya Math. J. \tb{81} (1981), 153--175.

\bibitem{Mum74} D. Mumford, \textit{Abelian varieties}, 2nd edition, Oxford University Press,  1974.

\bibitem{GIT} D. Mumford, J. Fogarty and F. C. Kirwan, \textit{Geometric Invariant Theory}, 3rd edition, Springer-Verlag, 1994.

\bibitem{Nak04} N. Nakayama, {\it Zariski decomposition and abundance}, MSJ Memoirs, \tb{14}. Mathematical Society of Japan, Tokyo, 2004.


\bibitem{Oda71} T. Oda, {\it Vector bundles on an elliptic curve}, Nagoya Math. J. \tb{43} (1971), 41--72.

\bibitem{PP03} G. Pareschi and M. Popa, \textit{Regularity on abelian varieties I},  J. Amer. Math. Soc. \tb{16} (2003), no. 2, 285--302.

\bibitem{PP11a} G. Pareschi and M. Popa, \textit{GV-sheaves, Fourier-Mukai transform, and generic vanishing}, Amer. J. Math. \tb{133} (2011), no. 1, 235--271.

\bibitem{Pa14} Z. Patakfalvi, \textit{Semi-positivity in positive characteristics}, Ann. Sci. \'{E}cole Norm. Sup. \textbf{47} (2014), 993--1025.

\bibitem{PSZ13} Z. Patakfalvi, K. Schwede and W. Zhang, \textit{$F$-singularities in Families}, Algebr. Geom. \tb{5} (2018), no. 3, 264--327.

\bibitem{Se58a} J.P. Serre, \textit{Queleques propri\'{e}t\'{e}s de vari\'{e}t\'{e}s ab\'{e}liennes en caract\'{e}ristique $p$}, Amer. J. Math. \textbf{80} (1958), 715--739.

\bibitem{Se58b} J.P. Serre, \textit{Sur la topologie des vari\'{e}t\'{e}s alg\'{e}briques en caract\'{e}ristique p}, Symposium
internacional de topologia algebraica, Mexico (1958), 24-53; Oeuvres, v. 1, 501--530.

\bibitem{Ta14} H. Tanaka, {\it Minimal models and abundance for positive characteristic log surfaces}, Nagoya Math. J. \tb{216} (2014), 1--70.

\bibitem{Ta15a} H. Tanaka, \textit{Behavior of canonical divisors under purely inseparable base
changes}, J. Reine Angew. Math. \tb{744} (2018),  237--264.

\bibitem{Ta15b} H. Tanaka, \textit{Abundance theorem for surfaces over imperfect fields}, Math. Z., (2019), https://doi.org/10.1007/s00209-019-02345-2.

\bibitem{Wal15} J. Waldron, {\it Finite generation of the log canonicial ring for 3-folds in char $p$}, Math. Res. Lett. \textbf{24} (2017), 933--946.

\bibitem{Xu15} C. Xu, \textit{On base point free theorem of three fold in positive characteristic}, J. Inst. Math. Jussieu \tb{14} (2015), no. 3, 577--588.

\bibitem{Zh16a} L. Zhang, \textit{Subadditivity of Kodaira dimensions for fibrations of three-folds in positive characteristics}, Adv. Math. \tb{354} (2019), arXiv: 1601.06907.

\bibitem{Zh16b} L. Zhang, {\it A note on Iitaka's conjecture $C_{3,1}$ in positive characteristics}, Taiwanese J. Math. \textbf{21} (2017), 689--704.


\bibitem{Zh16c} L. Zhang, \textit{Abundance for non-uniruled 3-fold with non-trivial Albanese map in positive characteristics}, J. Lond. Math. Soc. \textbf{99} (2019), no. 2, 332--348.
\end{thebibliography}
\end{document}